\documentclass[10pt]{amsart}
\usepackage{amssymb}
\usepackage{stmaryrd}
 \usepackage{graphicx}
 \usepackage{tikz, xcolor}
\usetikzlibrary {shapes}
 \usepackage{array}
\definecolor{dnrbl}{rgb}{0,0,0.5}
\definecolor{dnrgr}{rgb}{0,0.5,0}
\definecolor{dnrre}{rgb}{0.5,0,0}
\usepackage[colorlinks=true, citecolor=dnrgr, linkcolor=dnrre, urlcolor=dnrbl]{hyperref}
\usepackage{pdfsync}
\theoremstyle{plain}
\newtheorem{thm}{Theorem}[section]
\newtheorem{prop}[thm]{Proposition}
\newtheorem{lem}[thm]{Lemma}
\newtheorem{coro}[thm]{Corollary}
\newtheorem{defi}[thm]{Definition}
\numberwithin{equation}{section}
\newtheorem{question}{Question}

\newcommand{\la}{\langle}
\newcommand{\ra}{\rangle}

\newcommand{\Nat}{\mathbb{N}}

\newcommand{\restr}{\upharpoonright}  
\newcommand{\un}{\uparrow} 
\newcommand{\de}{\downarrow} 
\newcommand{\ml}{Martin-L\"{o}f }
\newcommand{\pz}{$\Pi^0_1$\ }

\begin{document}

\title[The typical Turing degree]{The typical Turing degree}

\author{George Barmpalias}
\address{{\bf George Barmpalias}, State Key Lab of Computer Science, 
Institute of Software, Chinese Academy of Sciences,
Beijing 100190, 
P.O. Box 8718,
People's Republic of China.}
\email{barmpalias@gmail.com}
\urladdr{\href{http://barmpalias.net}{http://barmpalias.net}}

\author{Adam R.~Day}
\address{{\bf Adam R.~Day}, University of California, Berkeley,
Department of Mathematics,
Berkeley, CA 94720-3840 USA.}
 \email{adam.day@math.berkeley.edu}

\author{Andrew E.M.~Lewis}
\address{{\bf Andrew E.M.~Lewis}, School of Mathematics,
University of Leeds, LS2 9JT Leeds, United Kingdom.}
 \email{andy@aemlewis.com}
 \urladdr{\href{http://aemlewis.co.uk/}{http://aemlewis.co.uk}}
 \date{4-11-2012}
\thanks{Barmpalias
 was supported by a
research fund for international young scientists
No.\ 611501-10168 and
an {\em International Young Scientist Fellowship} 
number 2010-Y2GB03 from the Chinese Academy of 
Sciences. Partial support was also obtained by
the {\em Grand project: Network Algorithms and Digital Information} of the
Institute of Software, Chinese Academy of Sciences.
Lewis was supported by a Royal
Society University Research Fellowship.  Day thanks the Institute for 
Mathematical Sciences of the National 
University of Singapore for support during June-July 2011}
\begin{abstract} The Turing degree of a real measures the computational 
difficulty of producing its binary 
expansion. Since Turing degrees are tailsets, it follows from Kolmogorov's 0-1 
law that for any property 
which may or may not be satisfied by any given Turing degree,  the satisfying 
class will either be of 
Lebesgue measure 0 or 1, so long as it is measurable.  So either the 
\emph{typical} degree satisfies the 
property, or else the typical degree satisfies its negation. Further, there is then 
some level of randomness 
sufficient to ensure typicality in this regard.  A similar analysis can be made in 
terms of Baire category, where 
a standard form of genericity now plays the role that randomness plays in the 
context of measure. 

   We describe and prove a number of results in a programme of research which 
aims to establish the 
properties of the typical Turing degree, where typicality is gauged either in terms 
of Lebesgue measure or 
Baire category.  
\end{abstract}
\maketitle
\setcounter{tocdepth}{1}
\tableofcontents

 \section{Introduction}\label{se:intro}
 The inspiration for the line of research which led to this paper begins essentially 
with Kolmogorov's 0-1 law, 
which states that any (Lebesgue) measurable tailset is either of measure 0 or 1. 
 The importance of this law for computability theory then stems from the fact that 
Turing degrees
\footnote{The Turing degrees were introduced by Kleene and Post in 
\cite{ Kleene.Post:54} and are a 
measure of the incomputability of an infinite sequence. For an introduction we 
refer the reader to 
\cite{Odifreddi:89} and \cite{Cooper:04}.} are clearly tailsets---adding on or taking 
away any finite initial 
segment does not change the difficulty of producing a given infinite sequence.  
Upon considering properties 
which may or may not be satisfied by any given Turing degree, we can 
immediately conclude that, so long 
as the satisfying class is measurable\footnote{By the measure of a set of Turing 
degrees is meant the 
measure of its union.}, it must either be of measure 0 or 1. Thus either the 
\emph{typical degree} satisfies the 
property, or else the typical degree satisfies its negation, and this suggests an 
obvious line of research. 
Initially we might concentrate on definable properties, where by a definable set of 
Turing degrees we mean 
a set which is definable as a subset of the structure in the (first order) language 
of partial orders. For each 
such property we can look to establish whether the typical degree satisfies the 
property, or whether it 
satisfies the negation. In fact we can do a little better than this. If a set is of 
measure 1, then there is some 
level of algorithmic randomness\footnote{The basic notions from algorithmic 
randomness will be described 
in Section \ref{techback}. For an introduction we refer the reader to 
\cite{Ottobook} and \cite{rodenisbook}.} 
which suffices to ensure membership of the set. Thus, once we have established 
that the typical degree 
satisfies a certain property, we may also look to establish the level of 
randomness required in order to 
ensure typicality as far as the given property is concerned.  
 
 Lebesgue measure though, is not the only way in which we can gauge typicality. 
One may also think 
in terms of Baire category. For each definable property, we may ask whether or 
not the satisfying class is 
comeager and, just as in the case for measure, it is possible to talk in terms of a 
hierarchy which allows us to 
specify  levels of typicality. The role that was played by randomness in the 
context of measure, is now played 
by a very standard form of genericity. For any given comeager set, we can look to 
establish the level of 
genericity which is required to ensure typicality in this regard. 
 
 \subsection{A heuristic principle}\label{subse:motiv}
During our research, we have
isolated the following heuristic principle:
{\em if a property holds for all
highly random/generic  degrees
then it is likely to hold for all non-zero degrees that are bounded by
a highly random/generic degree}. 
Here by `highly random/generic' we mean at least 2-random/generic
\footnote{The relevant forms of 
randomness, genericity and the corresponding hierarchies will be defined in 
section \ref{techback}.}. Thus, 
establishing levels of typicality which suffice to ensure satisfaction of a given 
property, also gives a way of 
producing lower cones and sets of degrees which are downward closed (at least 
amongst the non-zero 
degrees),  such that all of the degrees in the set 
satisfy the given property. 
For example,
by a simple analysis of a theorem of Martin 
\cite{Martin:60},
Kautz \cite{Kautz:91} showed that every 2-random degree is 
hyperimmune.\footnote{A degree is hyperimmune if it contains a function 
$f:\omega \rightarrow \omega$ 
which is not dominated by any computable function, i.e.\ such that for any 
computable function $g:\omega 
\rightarrow \omega$ there exist infinitely many $n$ with $f(n)>g(n)$. If a degree is 
not hyperimmune then we 
say it is hyperimmune-free. } In fact, this is just a special case of 
(\ref{eq:kauhgen}).
\begin{equation}\label{eq:kauhgen}
\parbox{11cm}{{\small Every non-zero degree that is bounded by
a 2-random degree is hyperimmune.}}
\end{equation}
  We may deduce (\ref{eq:kauhgen}) from certain facts
  that involve notions from algorithmic randomness. Fixing a universal prefix-free 
machine, we let $\Omega$ 
denote the halting probability.
  A set $A$ is called {\em low for} $\Omega$, if $\Omega$ is 1-random relative
  to $A$.  By \cite[Theorem 8.1.18]{Ottobook} every non-zero low
  for $\Omega$ degree is hyperimmune.
Since every 2-random real is low for $\Omega$
(a consequence of van Lambalgen's theorem, see 
\cite[Theorem 3.4.6]{Ottobook})
we have (\ref{eq:kauhgen}).

In this paper we will give several other examples that
support this heuristic principle. Moreover,  in
Section \ref{subse:allnzdbba} we give an explanation of the fact that it holds for 
the measure theoretic case,
by showing how to translate standard arguments which prove
that a property holds for all highly random degrees, into arguments that prove 
that the same property holds
for all non-zero degrees
that are bounded by a highly random degree.
The heuristic principle often fails for notions of
randomness that are weaker than 2-randomness and we
provide a number of counterexamples throughout this paper.
It is well known that the hyperimmunity 
example above fails for weak 2-randomness. However
Martin's proof in \cite{Martin:60} actually shows that every Demuth
random degree is hyperimmune. 
We shall give examples concerning minimality, the cupping property and the join 
property, which also 
demonstrate the principle for highly generic degrees. 

\subsection{The history of measure and category arguments in the Turing 
degrees}
Measure and Baire category arguments in degree theory are as old as the 
subject itself.
For example, Kleene and Post \cite{Kleene.Post:54} used arguments that 
resemble
the Baire category theorem construction in order to build Turing degrees with 
certain
basic properties. Moreover
de Leeuw, Moore, Shannon and Shapiro
 \cite{deleeuw1955} used a so-called `majority vote argument' 
 in order to show that
if a subset of $\omega$ can be enumerated relative to every set in a class of 
positive measure
then it has an unrelativised computable enumeration.
A highly influential yet unpublished manuscript by Martin 
\cite{Martin:60} showed that more advanced degree-theoretic results are possible 
using these
classical methods. By that time degree theory was evolving into a highly 
sophisticated subject
and the point of this paper was largely that category and measure can be used in
order to obtain advanced results, which go well beyond the basic methods
of \cite{Kleene.Post:54}. Of the two results in
\cite{Martin:60} the first was that the Turing upward closure of a meager set of 
degrees that is downward closed amongst the non-zero degrees, but which does 
not contain $\bf{0}$, is 
meager
(see \cite[Section V.3]{Odifreddi:89} for a concise proof of this).
Given that the minimal degrees form a meager class, an immediate corollary of 
this 
was the fact that there are non-zero degrees that do not bound minimal degrees.
The second result was that the measure of the hyperimmune degrees is 1.
Martin's paper was the main inspiration for much of the work that followed 
in this topic, including \cite{Yates:76}, \cite{paris77} and \cite{Jockusch:80}.

Martin's early work seemed to provide some hope that measure and category 
arguments could provide a 
simple alternative to conventional degree-theoretic constructions which are often 
very complex. This school 
of thought received a serious blow, however, with
\cite{paris77}. Paris answered positively a question of Martin which asked if the 
analogue
of his category result in \cite{Martin:60} holds for measure: are the degrees that 
do not bound minimal 
degrees of measure 1? Paris' proof was considerably more involved than
the measure construction in \cite{Martin:60} and seemed to require sophisticated 
new ideas.
The proposal of category methods as a simple alternative to `traditional' degree 
theory
had a similar fate.  Yates \cite{Yates:76} started working on a new approach
to degree theory that was based on category arguments and was even writing a 
book on
this topic. Unfortunately the merits of his approach were not appreciated at the
time (largely due to the heavy notation that
he used) and he gave up research on the subject altogether.

Yates' work in \cite{Yates:76} deserves a few more words, however, especially 
since it anticipated 
much of the work in \cite{Jockusch:80}.
Inspired by \cite{Martin:60}, Yates started a systematic
study of degrees in the light of category methods. 
A key feature in this work
was an explicit interest in the level of effectivity possible in the 
various category constructions and the translation of this level of effectivity into 
category concepts
(like `$\mathbf{0}'$-comeager' etc.).
Using his own notation and terminology, he studied the level of
genericity that is sufficient in order to guarantee that a set
belongs to certain degree-theoretic comeager classes, thus
essentially defining various classes of genericity already in 1974. He
analysed Martin's proof that the Turing upper closure  of a meager class which is 
downward closed amongst 
the non-zero degrees but which does not contain $\mathbf{0}$ is meager, for example 
(see \cite[Section 5]{Yates:76}),
and concluded that no 2-generic degree bounds a minimal degree.
Moreover, he conjectured (see \cite[Section 6]{Yates:76}) 
that there is a 1-generic that bounds 
a minimal degree.
These concerns occurred 
later in a more appealing form in  Jockusch \cite{Jockusch:80},
where simpler terminology was used and the hierarchy of $n$-genericity was
explicitly defined and studied.

With Jockusch \cite{Jockusch:80}, the heavy notation of Yates was dropped and 
a clear and systematic 
calibration of effective comeager classes
(mainly the hierarchy of $n$-generic sets)
and their Turing degrees was carried out. A number of interesting results were
presented along with a long list of questions that set a new direction for future 
research.
 The latter was followed up by 
Kumabe 
\cite{Kuma:90, Kuma:91, Kuma:93, apal/Kumabe93, Kuma:00} 
(as well as other authors, e.g.\ \cite{ChoDo90})
who answered a considerable number of these questions.

The developments in the measure approach to degree theory
were similar but considerably slower, at least in the beginning.
Kurtz's thesis \cite{Kurtz:81}
is probably the first systematic study of the Turing degrees
of the members of effectively large classes of reals, in the sense of
measure. Moreover the general methodology and 
the types of questions that Kurtz considers are
entirely analogous to the ones proposed in
\cite{Jockusch:80} for the category approach
(e.g.\ studying the degrees of the $n$-random reals
as opposed to the $n$-generic reals, minimality, 
computable enumerability and so on).
Ku\v{c}era \cite{MR820784} 
focused on the degrees of 1-random reals. 
Kautz \cite{Kautz:91} continued in the direction of 
\cite{Kurtz:81} but it was not until the last ten years
(and in particular with the writing of \cite[Chapter 8]{rodenisbook}) that
the study of the degrees of $n$-random reals became well
known and this topic became a focused research area. 

\section{Technical background, notation and terminology} \label{techback}
We let $2^{\omega}$ denote the set of infinite binary sequences
and denote the standard Lebesgue measure on $2^{\omega}$ by $\mu$. We 
let $2^{<\omega}$
denote the set of finite binary strings.
 We use the variables $c,d,e,i,j,k,
\ell,m,n,p,q,s,t$ to range over $\omega$; $f,g$ to range over functions 
$\omega \rightarrow \omega$; $\alpha,\beta,\sigma, \tau, \eta, \rho$ 
to range over $2^{<\omega}$; $A,B,C,D,X,Y,Z$ to range over $2^{\omega}$; we use 
$J,S,T,U,V,W$ to range over 
subsets of $2^{<\omega}$ and we use $F,G,P$ and $Q$ to range over subsets of 
$2^{\omega}$. We shall 
also use the variable $P$ to range over the various definable degree theoretic 
properties.  In the standard 
way we identify subsets of $\omega$ and their characteristic functions. 

\subsection{Turing functionals, Cantor space, strings and functions} 
For $\sigma\in 2^{<\omega}$ and $A\in 2^{\omega}$ we write
$\sigma\ast A$ to denote the concatenation of $\sigma$ and $A$,
and we say that $P\subseteq 2^{\omega}$ is a tailset if, for all
$\sigma\in 2^{<\omega}$ and all $A\in 2^{\omega}$, $\sigma\ast A\in P$ if
and only if $A\in P$.
A set $V \subseteq 2^{<\omega}$ is said to be {\em downward closed} if, 
whenever $ \tau \in V$, 
all initial segments of $ \tau $ are in this set, and is said to be {\em upward 
closed} if, whenever $ \tau \in
V $, all extensions of $ \tau $ are in this set. We write $\llbracket V\rrbracket$
to denote the set of infinite strings which extend some element of $V$, and we 
write $\mu(V)$ to denote $\mu(\llbracket V\rrbracket )$. 

    We use the variables $\Phi,\Psi, \Theta$ and $\Xi$ to range over the Turing 
functionals, and let $\Psi_i$ be 
the $i$th Turing functional in some fixed effective listing of all Turing functionals. 
Then $\Psi_i^{\sigma}(n)$ 
denotes the output of $\Psi_i$ given oracle input $\sigma$ on argument $n$. We 
make the assumption that 
$\Psi_i^{\sigma}(n)\uparrow$ unless the computation converges in $<|\sigma|$ 
steps and $\Psi_i^{\sigma}(n')\downarrow$ 
for all $n'<n$ (these assumption are also made for any 
{\em given} Turing functional $\Phi$, 
but we do not worry about adhering to these conventions when constructing 
Turing functionals). 
    Letting $\langle i,j \rangle$ be a computable bijection 
    $\omega \times \omega \rightarrow \omega$, we 
write $\omega^{[e]}$ to denote the set of all numbers of the form 
$\langle e,j \rangle$ for some $j\in \omega$. 
    
    To help with readability, we shall generally make some effort to maintain a 
certain structure in our use of 
variables. In situations in which we consider the actions of a functional, we shall 
normally use the variables 
$X$ and $\tau$ for sequences and strings in the domain, 
and the variables $Y$ and $\sigma$ for 
sequences and strings in the image. When 
another functional then acts on the image space, we shall generally use the 
variables $Z$ and $\eta$ for 
sequences and strings in the second image space. 
The variables $X,Y$ and $Z$ will generally 
be used in situations in 
which we are simultaneously dealing with all sets of natural numbers. When a 
specific set is given for a 
construction, or has to be built by a construction, then we will use the variables 
$A,B,C$ and $D$.

\subsection{Randomness and Martin-L\"{o}f tests} If each $V_i$ is a set of finite 
binary strings and the 
sequence $\{ V_i \}_{i\in \omega}$ is uniformly computably enumerable (c.e.), 
i.e.\ the set of all pairs $(i,\tau)
$ such that $\tau \in V_i$ is c.e., then we say that this sequence is a 
Martin-L\"{o}f test if $\mu(V_i)<2^{-i}$ for 
all $i$. Then we say that $X$ is Martin-L\"{o}f random if there doesn't exist any 
Martin-L\"{o}f test such that $X
\in \bigcap_i \llbracket V_i\rrbracket $. 
It is not difficult to show that there exists a {\em universal} Martin-L\"{o}f 
test, i.e.\ a Martin-L\"{o}f test 
$\{ V_i \}_{i\in \omega}$ such that $X$ is Martin-L\"{o}f random if and only if
$X\notin  \bigcap_i \llbracket V_i\rrbracket $. 

These notions easily relativize. 
We say that $\{ V_i \}_{i\in \omega}$ is a Martin-L\"{o}f test relative to $X$ if it 
satisfies the definition of a Martin-L\"{o}f test, except that now the sequence need 
only be uniformly c.e.\ 
relative to $X$. Now $Y$ is Martin-L\"{o}f random relative to $X$ if there does not 
exist any  Martin-L\"{o}f test 
relative to $X$ such that 
$Y\in \bigcap_i \llbracket V_i\rrbracket $. Once again, it 
can be shown that there exists a universal test relative to 
any oracle, and that, in fact, 
this universal test can be uniformly enumerated for all oracles. We let 
$\{U_i \}_{i\in \omega}$ be a uniformly 
c.e.\ sequence of operators such that, for any $X$, $\{ U_i^X \}_{i\in \omega}$  is 
a universal test relative to 
$X$. We assume that, for each $i$ and $\tau$, $U^{\tau}_i$ is finite, and is empty 
unless $|\tau|>i$. We 
assume furthermore, that the function $\tau\mapsto U_i^{\tau}$ is computable.

If a subset of Cantor space $P$ is of measure 1, then it is clear that there is 
some oracle $X$ such that all 
sets which are Martin-L\"{o}f random relative to $X$ belong to $P$. For 
$n\geq 1$ we say that $X$ is
 $n$-random if it is Martin-L\"{o}f random relative to 
 $\boldsymbol{0}^{(n-1)}$ 
(and that a degree is
 $n$-random if it contains an $n$-random set). 
 Martin-L\"{o}f randomness is in many respects the standard notion of algorithmic 
randomness.
Other randomness notions may be obtained by varying the level of computability 
in the above definition.
For example, a set is weakly $2$-random if it is not a member of any $\Pi^0_2$ 
null class.
In order to define Demuth randomness, we need to consider the wtt-reducibility. 
We say $X\leq_{wtt} Y$ if 
there exists $i$ such that 
$\Psi_i^X=Y$ and there exists a computable function 
$f$ such that the use on 
argument $n$ is bounded by $f(n)$. Let $W_i$ be the $i$th c.e.\  set of finite 
binary strings according to 
some fixed effective listing of all such sets. We say that $X$ is Demuth random if 
there is no $f$ which is wtt-
reducible to $\emptyset'$, such that 
$\mu(W_{f(i)})<2^{-i}$ and $X\in \llbracket W_{f(i)}\rrbracket$ 
for infinitely many $i$.    
Demuth randomness and weak 2-randomness are incomparable notions, both 
stronger than 1-randomness 
and weaker than 2-randomness.

\subsection{The $n$-generics} We say that $Y$ is 1-generic relative to $X$ if, for 
every $W\subseteq 2^{<
\omega}$ which is c.e.\ relative to $X$:

\[ (\exists \sigma \subset Y)[  \sigma \in W \ \vee\ (\forall \sigma' \supset \sigma)
(\sigma'\notin W) ] . \] 

\noindent It is clear that if a set $P$ is comeager then there is some oracle $X$ 
such that every set which 1-
generic relative to $X$ belongs to $P$.  For $n\geq 1$, we say that $Y$ is 
$n$-generic if it is 1-generic 
relative to $\boldsymbol{0}^{(n-1)}$,  and that a degree is $n$-generic if it 
contains an $n$-generic set.

\subsection{Jump classes} The generalized jump hierarchy is defined as follows. 
For $n\geq 1$ a Turing 
degree is generalized low$_n$ (GL$_n$), if 
$\boldsymbol{a}^{(n)}=(\boldsymbol{a} \vee \boldsymbol{0}')^{(n-1)}$, 
and we say that $\bf{a}$ is generalized high$_n$ (GH$_n$) if  
$\boldsymbol{a}^{(n)}=(\boldsymbol{a} \vee \boldsymbol{0}')^{(n)}$. 
A degree is generalized low if it is GL$_1$ and is generalized high if it is GH$_1$. 
A degree is low$_n$ if it is GL$_n$ and below $\boldsymbol{0}'$. A degree is 
high$_n$ if it is GH$_n$ and 
below $\boldsymbol{0}'$.  
By low is meant low$_1$ and by high is meant high$_1$.

\section{0-1 laws in category and measure}
In the analysis we have considered so far, we have left a gap which we now 
close. \emph{If} a tailset is 
measurable then it is either of measure 0 or 1, and there is then some level of 
randomness that suffices to 
ensure typicality. If we restrict to considering definable sets of Turing degrees, 
however (and where by 
definable we mean definable  in the first order language of partial orders), this 
begs the question, do all such 
sets have to be measurable? Similarly we may ask, do all such sets have to be 
either meager or comeager? 
In this section we make the following two observations, the second of which was 
made in an email 
correspondence with Richard Shore:

\begin{equation}\label{eq:meaindi}
\parbox{10cm}{Whether or not all definable sets of degrees
are measurable is independent of ZFC.}
\end{equation}

\begin{equation}\label{eq:cateindi}
\parbox{10cm}{Whether or not all definable sets of degrees
are either meager or comeager  is independent of ZFC.}
\end{equation}

 We consider first how to prove \ref{eq:meaindi}, the proof of \ref{eq:cateindi} will 
be similar. 
 On the one hand, all definable subsets of the Turing degrees are clearly 
projective. If we assume Projective 
Determinacy, then they will all be measurable \cite{MS64}. On the other hand, we 
wish to make use of the 
fact, due to Slaman and Woodin \cite{SW86}, that any set of Turing degrees 
above $\bf{0}''$ is definable as 
a subset of the Turing degrees if and only if its union is definable in second order 
arithmetic. Initially there 
might seem a basic obstacle to using this fact. We wish to construct a set which 
is of outer measure 1 and 
whose complement is also of outer measure 1. The degrees above $\bf{0}''$ are 
of measure 0, and so any 
subset will be measurable. It is easy to see, however, that the result of Slaman 
and Woodin extends to any 
set of degrees which is invariant under double jump---meaning that if $\bf{a}$ 
belongs to the set, then all 
$\bf{b}$ with $\bf{b}''=\bf{a}''$ are also members. Now, it is easy enough to 
construct a tailset which is of outer 
measure 1 and whose complement is also of outer measure 1, a result due to 
Rosenthal \cite{JR75}. One 
simply defines the set using a transfinite recursion which diagonalises against the 
open sets of measure 
$<1$. This recursion uses a well-ordering of the reals (which suffices to specify a 
well-ordering of the open 
sets). If we assume V=L then we have a well-ordering of the reals which is 
definable in second order 
arithmetic, and the set constructed will be definable in second order arithmetic. 
Finally we just have to 
modify the construction so as to make the set constructed invariant under double 
jump. This means that 
whenever we enumerate a real into the set or its complement, we also 
enumerate all reals which double 
jump to the same degree. Since we still add only countably many reals into either 
the set or its complement 
at each stage of the transfinite recursion, the argument still goes through as it did 
previously.   
 
 In order to prove \ref{eq:cateindi} we proceed in almost exactly the same way. In 
order to construct a 
definable set of degrees which is neither meager nor comeager, we consider this 
time a transfinite recursion 
which defines a set which does not satisfy the property of Baire (see \cite{AK95}, 
for example, for the 
description of such a construction).

\section{Methodology}\label{sec:method}
In this section we discuss a framework  for constructions which 
calculate the measure of a given degree-theoretic class.
By (\ref{eq:meaindi}) no methodology can be completely general, and as one 
moves to consider more 
complicated properties it is to be expected that more sophisticated techniques 
will be required. The 
methodology we shall present here, however, does seem to be very widely 
applicable. All previously known 
arguments of this type, and all of the new theorems we present here,  fit neatly 
into the framework.  An  
informal presentation of the framework is given
in Section \ref{subse:asrdsdf}.

Given a degree-theoretic property $P$ which holds for 
almost all reals, we consider (oracle-free)  constructions which work for all sets 
simultaneously and 
which specify a $G_{\delta}$ null set such that every real is either
in this set or satisfies $P$. By examining the oracle required to produce arbitrarily 
small open coverings of 
this  $G_{\delta}$ set, we establish a level of randomness
which is sufficient for a real to satisfy $P$. In all known examples it turns out
that 2-randomness suffices and, moreover,
that every non-zero degree
that is bounded by a 2-random also satisfies $P$.
A widely applicable methodology for results of the latter type is given in
Section \ref{subse:allnzdbba}.
In Section \ref{subse:meatoo} we give a number
of rather basic facts about measure in relation to Turing computations
that will be used routinely in most of the proofs in this paper.

Our framework rests on various ideas from 
\cite{Martin:60}, \cite{paris77} and \cite{Kurtz:81},
but introduces new features (like the use of measure density theorems)
which simplify and refine the classic arguments as well as
establishing new results in a uniform fashion.

\subsection{All sufficiently random degrees}\label{subse:asrdsdf}
The strategy for showing that all sufficiently random sets $X$ 
satisfy a certain degree-theoretic property is as follows:
\begin{itemize}
\item[(a)] Translate the property into a countable sequence
of requirements $ \{ R_e\}_{e\in \omega}$ referring to an unspecified set $X$.
\item[(b)] Devise an `atomic' strategy which takes
a number $e$ and a string $\tau$ as inputs
and satisfies $R_e$ for a certain proportion of extensions $X$
of $\tau$, where this proportion depends on $e$ and not on $\tau$.
\item[(c)] Assemble a construction from the atomic strategies in
a {\em standard way}.
\end{itemize}
Since steps (a) and (b) are specific to the degree-theoretic property that is 
studied,
we are left to give the details of the
procedure that produces the construction, given the
requirements and the corresponding atomic strategies.
Step (c) involves a construction that proceeds in stages
and places  `$e$-markers' (for $e\in\omega$) on
various strings in the full binary tree.
Each $e$-marker is associated with a version of the atomic strategy for $R_e$
from step (b), which looks to satisfy $R_e$ on a certain proportion of the 
extensions of the string $\tau$ on 
which it is placed. Once an $e$-marker is placed on $\tau$, we shall say that the 
marker `sits on' $\tau$ until 
such a point as it is removed. So a marker may be `placed on' $\tau$ at a specific 
point of the construction, 
and then at this and all subsequent points of the construction, until such a point 
as it is removed,  the marker 
is said to `sit on' $\tau$.  The basic rules according to which markers are placed 
on strings and removed
from them are as follows:
\begin{itemize}
\item[(i)] At most one marker sits on any string at any given stage. 
\item[(ii)] If $\tau\subset\tau'$ and at some stage an $e$-marker sits on $\tau'$ 
and a $d$-marker sits on 
$\tau$, then $d\leq e$.
\item[(iii)] If a marker is removed from $\tau$ at some stage
then any marker that sits on any extension of $\tau$ is also removed.
\end{itemize}
Note that (ii) and (iii) indicate an injury argument that is taking place
along each path $X$. A marker is called {\em permanent} if it is placed
on some string and is  never subsequently removed.
The basic rules above allow the possibility  that, for some
$e\in\omega$, many (perhaps permanent) $e$-markers are placed along a single 
path.
This corresponds to multiple attempts to satisfy $R_e$ along the path.

The construction will strive to address each requirement $R_e$ along 
the `vast majority' of the paths $X$ of the binary tree.
In particular, it will work with an arbitrary parameter $k\in\omega$
and will produce the required objects (like various reductions
that are mentioned in the requirements) along with a set of
strings $W$ such that $\mu(W)<2^{-k}$.
{\em Every real that does not have a prefix in $W$ will satisfy
all $R_e, e\in\omega$}. Considering all of the constructions
as $k$ ranges over $\omega$, we conclude that $P$ is satisfied by  every real 
except for those in a certain 
null 
$G_{\delta}$ set. Since this set may be seen as a \ml test relative to some oracle,
we can also establish a level of randomness that is sufficient to guarantee 
satisfaction of the  property. This is 
directly related to the 
oracle that is needed for the enumeration of $W$. In all of our
examples an oracle for $\emptyset'$ suffices to enumerate $W$, and thus 
2-randomness is sufficient to ensure 
satisfaction of $P$.
In most of our examples we will be able to show that
any standard weaker notion of randomness (in particular, weak
2-randomness) fails to be sufficient.

The outcome of the construction with respect to a particular real $X$
will be reflected by the permanent markers that are placed on initial segments of 
$X$.
In particular, one of the following outcomes will occur:
\begin{itemize}
\item[(1)] For every  $e\in\omega$ there is a permanent $e$-marker placed on 
some initial segment of $X$. 
\item[(2)] There exists some $e\in\omega$ such that, for each $d\leq e$,  a 
permanent $d$-marker is 
placed on an initial segment of $X$,  and such that infinitely many 
permanent $e$-markers 
are placed on initial segments of $X$. 
\item[(3)] There are only finitely many
permanent markers placed on initial segments of $X$.
\end{itemize}
Note that by rule (ii), if outcome (2) occurs with respect to $X$ then  for $j>e$
there will be no permanent $j$-marker placed on any initial segment of $X$, and 
for each
$d<e$ there will only be finitely many (permanent or non-permanent)
 $d$-markers placed on initial segments of $X$.

The only successful outcome for $X$ is (1).
Failure of the construction with respect to $X$ therefore comes in two forms.
Outcome (3) denotes a {\em finitary} failure.  In this case the construction gives 
up placing markers on initial 
segments of $X$,
due to the request of an individual marker that sits on
an initial segment $\tau$ of $X$.
Such a marker may forbid
the placement of markers on certain extensions of $\tau$
(including a prefix of $X$), while waiting for some $\Sigma^0_1$ event.
\footnote{As an example, this event 
might be the convergence of a computation which, should it be found, would then 
allow the marker to effect 
a successful diagonalisation above all those strings where it has previously 
paused the construction (in 
effect) by forbidding the placement of markers.}   
At any
stage during the construction,
requests to forbid the placement of markers will only be made for a 
small measure of sets, and 
so we will be able define a set of strings $V$ of small measure, such that every 
real for which outcome (3) 
occurs has an initial segment in $V$. 

Outcome (2) denotes an {\em infinitary} failure, in the sense that
the construction insists on trying to satisfy a certain 
requirement $R_e$ with respect to  $X$ by placing infinitely many $e$-markers
on initial segments of it, but the requirement remains unsatisfied with respect to 
$X$.
The possibility of outcome (2) is a direct consequence of (b), which says that
the atomic strategy only needs to satisfy the requirement
on a fixed (possibly small) proportion of the reals in its neighbourhood
(leaving the requirement unsatisfied on many other reals).
Reals for which outcome (2) occurs, are those which happen to always be in the 
unsatisfied part of the 
neighbourhood that corresponds to each $e$-marker.
The Lebesgue density theorem tells us, however,  that 
the reals for which outcome (2) occurs cannot form a class of positive measure. 
In particular it tells us that,  
for almost all reals in this class, the limit density must be 1. The existence of an 
element of the class for 
which the limit density is 1 contradicts the fact that (b) insists  the requirement be 
satisfied for a fixed 
proportion of strings extending that on which the marker is placed.   This class 
therefore has measure 0, and 
we can
consider a set of strings $S$ of arbitrarily small measure 
which contains a prefix of every real
in the class.
Then we can simply let $W$ be the union of $V$ and $S$. 

Such constructions will typically be computable, thus constructing
Turing reductions dynamically. Hence the reals for which outcome (2) occurs 
will typically form a  $\Sigma^0_3$
class and $V$ and $S$ will usually require an oracle for $\emptyset'$ for
their enumeration. This is the reason that 2-randomness is required
in all of the results that involve this type of construction.

\subsection{Measure theoretic tricks concerning Turing reductions}
\label{subse:meatoo}
Given a Turing functional $\Psi$, if we are only interested in computations
that $\Psi$ performs relative to a `sufficiently random' (typically a
2-random) oracle, then we
can expect certain features from $\Psi$.
This section discusses features which are particularly useful 
for the arguments employed in this paper. Section
\ref{sssec:tpori} shows that  we may assume
 all infinite binary sequences in the range of  $\Psi$ are incomputable.
In Section \ref{subse:splitfu} we describe a basic fact concerning the measure of
the splittings which can be expected to exist for such a functional $\Psi$
(a tool that is essential in certain coding arguments, including 
the one in Section \ref{sec:joinprop}).
Finally, in Section \ref{subse:ddmeth}
we give a $\Psi$-analogue of the Lebesgue density theorem
which will be an essential tool for extending results
to nonzero degrees below a 2-random degree. 

\subsubsection{Turing procedures on random input}\label{sssec:tpori}
We start with the following useful fact, which says that each Turing functional 
$\Phi$ can be replaced with one
which restricts the domain to sequences $X$ which $\Phi$-map to sets relative to 
which $X$
is not random.

\begin{lem}[Functionals and relative randomness]\label{le:oprelran}
For each Turing functional $\Phi$ there is
a Turing functional $\Psi$ which satisfies the following for all $X$:
\begin{enumerate}
\item[(a)] If $\Psi^X$ is total then $\Phi^X$ is total,  $\Psi^X=\Phi^X$ and $X$ is 
not $\Psi^X$-random.\footnote{By $Y$-random is 
meant Martin-L\"{o}f random relative to $Y$.}
\item[(b)] If $\Phi^X$ is total and $X$ is not $\Phi^X$-random then 
$\Psi^X$ is total. 
\end{enumerate} 
Moreover, an index for $\Psi$ can be obtained effectively from an index for 
$\Phi$.
\end{lem}
\begin{proof}
We describe how to enumerate axioms for $\Psi$, given the functional $\Phi$.
Let $ \{ U_i \}_{i\in \omega}$ be a universal oracle test as described in Section 
\ref{techback}. 
At stage $s$, for each pair of strings $\tau$, $\sigma=\rho \ast j$ 
of length $<s$, if $i$ is
the least number such that $\tau$ does not extend any string in $U_i^{\rho}$ 
then do the following.
If $\Phi^{\tau}\supseteq \sigma$ and 
$\tau$ extends a string in $U_{i}^{\sigma}$ then enumerate
the axiom $\langle \tau, \sigma \rangle$ for $\Psi$ 
(thus defining $\Psi^{\tau}\supseteq \sigma$).

Clearly $\Psi$ is obtained effectively from $\Phi$. If $\Psi^X$ is total for some
oracle $X$ and $\Psi^X=Y$, then $\Phi^X$ is also total and equal to $Y$. We 
also claim that in this case
$X\in U_i^Y$ for each $i\in\omega$. Towards a contradiction suppose that $i$ is 
the least number such that
$X\not\in U_i^Y$. If $i>0$ then let $\tau\subset X$ and $\sigma=\rho \ast j$ be 
such that $\tau$ does not 
extend any string in $U_{i-1}^{\rho}$, but does extend a string in 
$U^{\sigma}_{i-1}$, and such that we 
enumerate the axiom $\langle \tau,\sigma \rangle$. Let $s$ be the stage at which 
this  axiom is enumerated.  
If $i=0$ then let $s=0$. Then, subsequent to stage $s$ we do not enumerate any 
new axioms of the form $
\langle \tau',\sigma' \rangle$ such that $\tau' \subset X$. This gives us the 
required contradiction and 
concludes the verification of property (a).
For (b) suppose that $\Phi^X=Y$ and that $X\in U_i^Y$ for all $i\in\omega$. 
Then, since it cannot be the 
case  for any finite string $\sigma$ that $X\in U_i^{\sigma}$ for all $i$ (according 
to the conventions established 
in Section \ref{techback}),  it follows that $\Psi^X$ is total.
\end{proof}
\noindent
In most measure arguments in this paper we will
use Turing functionals which do not map to
computable reals. This will simplify
the constructions. 
\begin{defi}[Special Turing functionals]
A Turing functional $\Psi$ is called special if all infinite strings in the range are 
incomputable.  
\end{defi}
\noindent
The following lemma (when combined with the fact that any non-empty $\Pi^0_1$ 
class containing only 1-
randoms contains a member of every 1-random degree)  will be used throughout 
this paper
in order to justify the use of special functionals in various
arguments which involve given reductions.
\begin{lem}[Obtaining special functionals]\label{coro:replfunat}
Given a Turing functional $\Phi$ 
and a non-empty \pz class $P$
which contains only 1-random sequences
we can effectively obtain 
a special Turing functional $\Psi$ which satisfies the following conditions for  
every 2-random set $X$ in $P
$:
\begin{enumerate}
\item[(i)] If $\Psi^X$ is total then 
$\Phi^X$ is total and  $\Psi^X=\Phi^X$.
\item[(ii)] If $\Phi^X$ is total and incomputable 
then $\Psi^X$ is total. 
\end{enumerate} 
\end{lem}
\begin{proof}
Let $V$ be a c.e.\ set of finite strings such that a real is in
$P$ if and only if it does not have a prefix in $V$.
Given $V$ and $\Phi$ we produce $\Psi$ as in
the proof of Lemma \ref{le:oprelran} with the additional clause that
whenever a string $\tau$ appears in $V$ at some
stage of the construction, we stop
enumerating axioms for $\Psi$ of the form $\langle \tau',\sigma'\rangle$ such that 
$\tau'$ extends $\tau$. 

Let $X$ be a 2-random member of $P$.
Clearly $\Psi^X$  satisfies (a) and (b) of
Lemma \ref{le:oprelran}.
This shows (i) above.
For (ii), we need a notion from \cite{MR1238109}:
a set is called a {\em basis for 1-randomness} 
if there is a set that computes it and is 1-random
relative to it. 
By \cite{MR2352724},
bases for 1-randomness are $\Delta^0_2$.
On the other hand no 2-random set computes
an incomputable $\Delta^0_2$ set.
Hence 2-random sets do not bound incomputable
bases for 1-randomness and (ii) follows from (b)
of Lemma \ref{le:oprelran}.

Finally we show that $\Psi$ is special.
If $\Psi^X$ is total then $X$ must be a member
of $P$. Therefore it is 1-random. By (a)
of Lemma \ref{le:oprelran}, totality of $\Psi^X$ means that $X$ is not 
$\Psi^X$-random. This shows that $\Psi^X$ is incomputable. 
\end{proof}
\noindent
The use of special functionals
in what follows is not necessary but it 
often simplifies the proofs considerably. The simplification comes from the fact 
that the use of special 
functionals will often 
reduce the number of outcomes that a strategy has.
The following fact is applicable in arguments
where we show that some property holds for all non-zero degrees
below a sufficiently random degree.
\begin{lem}[Special functionals for downward density]\label{coro:replfungenat}
Given Turing functionals $\Theta, \Phi$ 
and a non-empty \pz class $P$
which contains only 1-random reals
we can effectively produce a special Turing functional $\Psi$ 
 which satisfies the following conditions for  every 2-random set $X$ in $P$:
\begin{enumerate}
\item[(a)] If $\Psi^Y$ is total for any $Y$, then it is equal to $\Phi^Y$.
\item[(b)] If $\Theta^{X}=Y$ and $\Phi^Y$ is total and incomputable
then $\Psi^Y$ is total.
\end{enumerate} 
\end{lem}
\begin{proof}
We describe how to enumerate the axioms for $\Psi$.
Let $V$ be an upward
closed computable set of strings which
contains  initial segments of precisely those reals which are not in $P$.
At stage $s$, for each triple $\tau, \sigma$, $\eta=\rho\ast j$ 
such that all strings in the triple are of length $<s$ and such that $\tau \not\in V$, 
if $i$ is
the least number such that $\tau$ does not extend a string in $U_i^{\rho}$ 
then do the following.
If $\Theta^{\tau}=\sigma$, $\Phi^{\sigma}\supseteq  \eta$ and 
$\tau$ extends a string in $U_{i}^{\eta}$ then enumerate
the axiom $\langle \sigma, \eta \rangle$ for $\Psi$. 

Clearly (a) holds. If $\Psi^Y$ is total then there is some $X\in P$ such that
$\Theta^X=Y$, $\Phi^Y=\Psi^Y$ and $X$ is not random relative to $\Psi^Y$.
Hence $\Psi^Y$ is incomputable, and thus $\Psi$ is
special. 
For (b) suppose that $\Theta^X=Y$
for some 2-random $X$ which is in 
$P$ such that $\Phi^Y$ is total and incomputable.
Then $X$ is not random relative to $\Phi^Y$ because 2-random
reals do not compute incomputable bases for 1-randomness. Therefore
the construction will define $\Psi^Y=\Phi^Y$.
\end{proof}

\subsubsection{Measure splittings for Turing functionals}\label{subse:splitfu}
Recall that a {\em $\Psi$-splitting} is
a pair of strings $\tau,\tau'$ 
such that  $\Psi^{\tau}$ and $ \Psi^{\tau'}$
are incompatible.
When we deal with functionals that operate on a random oracle,
a measure theoretic version of this notion is useful.
\begin{equation}\label{eq:sigmamea}
\parbox{11cm}{Given a set of reals $X$ and a string $\tau$,
the $\tau$-measure of $X$ 
is the measure of the reals in $X$ with prefix $\tau$, multiplied
by $2^{|\tau|}$.}
\end{equation}
 Given a Turing functional $\Psi$, a string 
 $\tau$ and a real number $\epsilon$
  we say 
 that a pair $(U,V)$ of finite sets of strings is a 
 {\em $\Psi$-splitting above $\tau$}
 if:
\begin{itemize}
\item the strings in $U\cup V$ all have the same length and extend $\tau$;
\item if $\tau_0\in U$ and $\tau_1 \in V$ then $\tau_0$ 
and $\tau_1$ are $\Psi$-splitting.
\end{itemize}
\noindent
Moreover, we say that $(U,V)$ has
  measure $\epsilon$ if
  $\mu(U)= \mu(V)= \epsilon/2$.
 A rational number is {\em dyadic} if it has a finite binary expansion.
 We define:
\begin{equation}\label{eq:pidefi}
\pi(\Psi,\sigma)=\mu(\{X\ |\ \Psi^X\supseteq\sigma\}).
\end{equation}
If $U$ is a prefix-free set of strings
and $\Psi$ is a functional then we let $\pi(\Psi, U)$ be the sum of all
$\pi(\Psi, \sigma)$ for $\sigma\in U$.

 \begin{prop}\label{prop:psigotz}
If $\Psi$ is a special Turing functional
then for each $c\in\omega$ and each $\sigma$ there exists $\ell\in\omega$ such 
that 
$\pi(\Psi,\sigma')/\pi(\Psi,\sigma)\leq 2^{-c}$ for all  $\sigma'\supset \sigma$ of 
length $\ell$.
\end{prop}
\begin{proof}
For a contradiction, suppose that there exists some $c\in\omega$ such that for 
each
$\ell\in\omega$ we have $\pi(\Psi,\sigma')/\pi(\Psi,\sigma)>2^{-c}$ for some string 
$\sigma'\supset\sigma$ of 
length $\ell$.
Then by K\"{o}nig's lemma there exists an 
infinite binary sequence $Y$ extending 
$\sigma$
such that $\pi(\Psi, Y\restr_n)/\pi(\Psi,\sigma)>2^{-c}$ for all $n\in\omega$. This 
implies that
$Y$ is computable. For each $n$ there exists a clopen set $V_n$ such that 
$\mu(V_n)/\pi(\Psi,\sigma)>2^{-
c-1}$, such that all strings in $V_n$ $\Psi$-map to extensions of $Y\restr_n$ and 
such that $V_{n
+1}\subseteq V_n$. By compactness it follows that $Y$ is in the range of $\Psi$, 
which
contradicts the fact that $\Psi$ is special.
\end{proof}

A basic fact from classical computability theory is that if
 some oracle $X$ computes an incomputable set via a
 Turing reduction $\Psi$ then $\Psi$-splittings are dense
 along $X$. In other words, for every initial segment $\tau$ of $X$
 there exists a $\Psi$-splitting such that all strings in the splitting extend $\tau$.
 The measure theoretic version of this fact is as follows.

\begin{lem}[Measure splittings for functionals] \label{le:measplitbfa} 
Suppose that $\Psi$ is a special Turing functional, 
$\epsilon$ is a dyadic rational and $\tau$ is a string. If 
there does not exist a $\Psi$-splitting 
above $\tau$ of measure $\epsilon$ then
there exists a c.e.\ set $V$ of strings extending $\tau$  
such that
$\mu(V)\leq 2\epsilon$ and
every set extending $\tau$ on which $\Psi$ is total has a prefix in
$V$.
Moreover, given $\tau, \Psi$ and $\epsilon$, 
an oracle for $\emptyset'$ can find 
whether or not there exists such a splitting 
and, if there does not then  an index for $V$.
\end{lem}
\begin{proof} 
Let $\ell$ be the least number such that
$\pi(\Psi, \sigma)\leq \epsilon/2$ for all 
strings $\sigma$
of length $\ell$.
If the measure of all  $X\supset \tau$ such that
$|\Psi^{X}|\geq \ell$  is 
 greater than $2\epsilon$ then
there exists a 
$\Psi$-splitting 
above $\tau$  of measure $\epsilon$.
Otherwise we can let $V$ be the c.e.\ set of strings
$\tau'\supset\tau$ such that $|\Psi^{\tau'}|\geq \ell$.
Finally note that the above procedure
only involves $\Sigma^0_1$ questions,
and so can be carried out using an oracle for $\emptyset'$.
\end{proof}
\noindent
The following version of Lemma \ref {le:measplitbfa} 
is applicable in arguments
where we show that some property holds for all non-zero degrees
below a sufficiently random degree.

\begin{lem}[Measure splittings for downward density] \label{le:mediwasplitbfa} 
Suppose that $\Theta, \Psi$ are special Turing functionals, 
$\epsilon$ is a rational number and $\sigma$ is a string. If 
there does not exist a $\Psi$-splitting $(U,V)$
above $\sigma$ such that $\pi(\Theta, U)$ and $ \pi(\Theta, V)$
are at least $\epsilon/2$ then
there exists a c.e.\ set $V$ of strings  
such that
$\mu(V)\leq 2\epsilon$ and
every set which $\Theta$-maps to an extension of $\sigma$
on which $\Psi$ is total has a prefix in
$V$.
Moreover given $\sigma, \Theta, \Psi$ and $\epsilon$, 
an oracle for $\emptyset'$ can find 
whether or not there exists such a splitting 
and, if there does not then an index for $V$.
\end{lem}

\begin{proof} 
Let $\ell$ be the least number such that
$\pi(\Psi\circ\Theta, \eta)\leq \epsilon/2$ for all $\eta$ 
of length $\ell$.
If the measure of all reals which $\Theta$-map to
extensions of any $\rho\supset \sigma$ such that $|\Psi^{\rho}|\geq \ell$ 
is $>2\epsilon$ then
there exists a 
$\Psi$-splitting $(U,V)$ 
above $\sigma$  such that $\pi(\Theta, U)$ and 
$\pi(\Theta, V)$ are at least $\epsilon/2$.
Otherwise we can let $V$ be the c.e.\ set of strings
$\tau$ such that $\Theta^{\tau}$ extends $\sigma$
which $\Psi$-maps to a string of length $\geq \ell$.
Finally note that we only ask  $\Sigma^0_1$ questions,
so the above can be done computably in $\emptyset'$.
\end{proof}

\subsubsection{Measure density for Turing reductions}\label{subse:ddmeth}
The observations in this section are mainly to be applied in the
methodology that is described in Section \ref{subse:allnzdbba}.

\begin{lem}[$\Psi$-totality]\label{le:cconran}
Let $\Psi$ be a Turing functional, 
$c\in\omega$ and let $E$ be a set of tuples $(\sigma,\ell)$ such that the strings
occurring in the tuples form a prefix-free set  and for each $(\sigma,\ell)\in E$:
\begin{equation}\label{eq:conamasmea}
\mu(\{X\ |\ \sigma \subseteq \Psi^X\ \wedge\ |\Psi^X|\geq \ell\}) <  
2^{-c}\cdot \pi(\Psi, \sigma).
\end{equation}
Then the class of reals $X$ such that  a prefix of $\Psi^X$ occurs in some tuple 
$(\sigma,\ell) \in E$ and
$|\Psi^X|\geq \ell$,
has measure $<2^{-c}$.
\end{lem}
\begin{proof}
For each $(\sigma, \ell)\in E$ consider the set $M_{\sigma}$ of reals $X$ such 
that $\Psi^X\supseteq\sigma$.
The sets $M_{\sigma}$ are pairwise disjoint. Moreover, the proportion of the reals 
$X$ in $M_{\sigma}$
with $|\Psi^X|\geq \ell$ is $<2^{-c}$. 
Therefore the class of reals $X$ such that  a prefix of $\Psi^X$ occurs in some 
tuple $(\sigma,\ell) \in E$ and
$|\Psi^X|\geq \ell$,
has measure $<2^{-c}$.
\end{proof}
\noindent
 Finally we give an analogue of the Lebesgue density theorem
 which refers to a Turing functional $\Theta$ and a set of strings $V$.
It says that if $F$ consists of the
reals $X$ for which $\Theta^X$ is total and does not have 
a prefix in $V$, then
for almost all $X\in F$ the proportion of the reals that
$\Theta$-map to $\Theta^X\restr_n$ which are in $F$ tends to
1 as $n\to\infty$.
 \begin{lem}[$\Theta$-density]\label{le:extled}
 Suppose $\Theta$ is a Turing functional,
$V$ is a set of finite strings and let $F_V$ be the set
of reals $X$ such that $\Theta^X$ is total and does not 
extend any strings in $V$. Then:
 \begin{equation}\label{eq:limi1exsle}
 \lim_n \frac{\mu\{X_1\in F_V\ |\ \Theta^{X_1}\supseteq \Theta^{X_0}\restr_n\}}
{\pi(\Theta^{X_0}\restr_n)}=1\ \ 
\textrm{for almost all $X_0\in F_V$},
 \end{equation}
 where $\pi(\sigma)=\pi(\Theta,\sigma)$ and `almost all' means `all but a set of 
measure zero'.
 \end{lem}
 \begin{proof}
Without loss of generality we may assume that $V$ is prefix-free.
For each $\epsilon\in (0,1)$ define: 
\begin{equation}\label{eq:uiinclude}
G_{\epsilon}=\{X_0\in F_V\ |\ \liminf_n \frac{\mu\{X_1\in F_V\ |\ 
\Theta^{X_1}\supseteq \Theta^{X_0}\restr_n\}}{\pi(\Theta^{X_0}\restr_n)}<
1-\epsilon\}.
\end{equation}
It suffices to show that for each $\epsilon\in (0,1)$ there exists a
sequence $Q_0\supseteq Q_1\supseteq\dots$ of open sets such that
$G_{\epsilon}\subseteq Q_i$
and $\mu(Q_{i+1})\leq \mu(Q_i)\cdot (1-\epsilon)$ for all $i\in\omega$.
Indeed, in that case we have $\lim_i\mu (Q_i)=0$ and 
so the reals $X_0$ in $F_V$ that fail (\ref{eq:limi1exsle})
form a null set.
For each $i$ we will define a set of string/number tuples $E_i$ and define: 
\[
Q_i=\{X\ |\ \Theta^X\ \textrm{has a prefix in a tuple of $E_i$}\}.
\]
Let $E$ be the set of tuples $(\sigma,\ell)$ such that $\ell>|\sigma|$, 
$\pi(\sigma)>0$ and
the proportion of the reals $X$ with $\Theta^X \supseteq \sigma$, such that 
either $\Theta^X\restr\ell$ is 
undefined or has a prefix in
$V$,  is $\geq \epsilon$.
We order the strings first by length  and then lexicographically. Also, we order 
$E$ lexicographically, i.e.\ 
$(\sigma,m)<(\sigma',n)$ when either $\sigma<\sigma'$, or $\sigma=\sigma'$ 
and $m<n$.

At step $i=0$ we define a sequence of tuples 
by recursion: let $(\sigma_j, \ell_j)$ 
be the least tuple in $E$ 
such that $\sigma_j$ 
is incompatible with
$\sigma_k$ for  $k<j$. Let $E_0$ be the collection of all these tuples.
At step $i+1$ do the following for each string $\sigma$ which does not have a 
prefix in $V$ and 
such that $|\sigma|=\ell$ and
$\sigma'\subseteq\sigma$ for some $(\sigma',\ell)\in E_i$.
Define a sequence of tuples by recursion, letting $(\sigma_j', \ell_j')$ be the least 
tuple in $E$ such that $\sigma_j'$
extends $\sigma$ and 
is incompatible with
$\sigma_k'$ for  $k<j$. Let $E_{i+1}$
 be the set of all tuples which occur in any  sequence defined at step $i+1$ 
 (i.e.\ take the union of all the 
sequences produced for the various $\sigma$ such that $\sigma$ does not have 
a prefix in $V$, $|\sigma|=\ell$ and
$\sigma'\subseteq\sigma$ for some $(\sigma',\ell)\in E_i$).

It follows by induction on $i$ that the set of all strings which are in any tuple in 
$E_i$ is prefix-free, and that 
$Q_i\supseteq Q_{i+1}$.
By the definition of $Q_0$ and the minimality of the strings that are enumerated 
into $E_0$ 
we have $G_{\epsilon}\subseteq Q_0$.
For the same reason, at each step $i+1$ we have $Q_i-Q_{i+1}\subseteq 
2^{\omega}-G_{\epsilon}$.
Hence $G_{\epsilon}\subseteq Q_i$ for all $i\in\omega$.
It remains to show that
$\mu(Q_{i+1})\leq \mu(Q_i)\cdot (1-\epsilon)$ for all $i\in\omega$.
In order to see this note that, at stage $i+1$, we consider in effect a partition of 
$Q_i$ into sets $Q_{\sigma}$ 
where $\sigma$ occurs in a tuple in $E_i$ and 
$Q_{\sigma}= \{X\ |\ \Theta^X\supseteq \sigma \}$. According 
to the definition of $E$, we only 
enumerate into $Q_{i+1}$ at most $1-\epsilon$ of 
the measure in each 
$Q_{\sigma}$. 
 \end{proof}

\subsection{Example: bounding a 1-generic degree}\label{se:bound1gendeg}
In this section we demonstrate how to apply the methodology that
was discussed in Section \ref{sec:method} by giving a 
 simple proof of 
a result from \cite{Kurtz:81} and \cite{Kautz:91}
that says that every 2-random degree bounds a 1-generic degree.
This result is also discussed in \cite[Section 8.21]{rodenisbook}.
This is the only level of genericity and randomness where the two notions
interact in a non-trivial manner. In fact, it follows from the results 
in this paper that 
{\em every 2-generic degree forms a minimal pair with every 2-random degree}.

\begin{thm}[Kurtz \cite{Kurtz:81} and 
Kautz \cite{Kautz:91}]\label{th:kur2ranb1gen}
Every 2-random degree bounds a 1-generic degree.
\end{thm}
\begin{proof}
Let $\{ W_e \}_{e\in \omega}$ be an effective enumeration of all
c.e.\ sets of finite binary strings. 
It suffices to define a computable procedure which takes $k\in\omega$
as input and returns the index of a $\emptyset'$-c.e.\ set of strings $W$ with 
$\mu(W)<2^{-k}$ and a 
functional $\Phi$
such that $\Phi^X$ is total and the following condition is met for all  $e\in\omega$ 
and each $X$ which does 
not have a prefix in $W$:
\[
R_e: \ \exists n\ \big[\Phi^X\restr_n\in W_e\ \vee\ \forall \sigma\in W_e,\ \Phi^X
\restr_n \not\subseteq \sigma
\big].
\]

\paragraph{{\bf Construction}} At stage $s+1\in 2\omega^{[e]}+1$, if  $e>k+1$ 
do the following. 
\begin{enumerate}

\item For each $e$-marker that has not acted and sits on a string $\tau$, if 
 $\Phi^{\tau}[s]=\sigma$ 
 and there is a proper extension $\rho$ of $\sigma$ in $W_e[s]$
then enumerate the axiom $\langle \tau\ast 0^e, \rho \rangle$ for $\Phi$,  and
declare that the marker has {\em acted}. 

\item  Let $\ell$ be large. For each string $\tau$ of length $\ell$ check to see 
whether there is 
 some $\tau'\subset \tau$ such that either (a) an $e$-marker sits on $\tau'$ and 
has not acted,  (b) an $e$-
marker sits on $\tau'$ that has acted and $\tau' \ast 0^e \subseteq \tau$, or (c) for 
some $i<e$ an $i$-marker 
sits on $\tau'$ that has not acted and   
$\tau'\ast 0^i\subseteq\tau$.  If none of these conditions hold then 
place an $e$-marker
on $\tau$ and remove any 
$j$-marker that sits on any initial segment of $\tau$ for $j>e$.

\end{enumerate} 

At stage $s+1\in 2\omega$ let
$\ell$ be large and for each $\tau$ of length $\ell$ enumerate the axiom  
$\langle \tau, \Phi^{\tau}[s]\ast 0 \rangle $ for $\Phi$ 
unless there is
some $e\in\omega$ and a string $\tau'$ 
with an $e$-marker sitting on it which has not acted, 
such that $\tau'\ast 0^e \subseteq \tau$.

\ \paragraph{{\bf Verification}} We start by noting that the axioms enumerated for 
$\Phi$ are consistent. Indeed, the only point at which an inconsistency could 
possibly occur is during step 
(1) of an odd stage $s+1$. During this step, when we enumerate 
an axiom $\langle \tau\ast 0^e, \rho\rangle$, 
$\rho$ extends $\Phi^{\tau}[s]$, and we have not enumerated any axioms with 
respect to proper 
extensions of $\tau$ which are compatible with $\tau \ast 0^e$. 

We consider versions of the outcomes (1)--(3), as described in Section 
\ref{sec:method},  which are modified 
to consider only $e>k+1$ in the obvious way. For each $e>k+1$ let $V_e$ be the 
set of strings
on which we place a permanent $e$-marker that never acts. 
When such a marker is placed on $\tau$ 
the construction will cease
placing $e$-markers on extensions of $\tau$, and  $V_e$ is therefore  
 prefix-free.  If we let
$V=\bigcup_{e>k+1} \{\tau\ast 0^e\ |\ \tau \in V_e\}$
then $\mu(V)\leq 2^{-k-1}$ and $V$ is c.e.\ in $\emptyset'$. This deals with the 
reals for which outcome (3) 
occurs.

Let $Q_e$ be the set of $X$ such that we place infinitely many $e$-markers on 
initial segments of $X$,  but 
finitely many $d$-markers
for each $d<e$.
If $X\in Q_e$ then all but finitely many of the $e$-markers placed on initial 
segments of $X$ will
be permanent and will act at some stage.
We claim that the measure of 
$Q_e$ is 0. If it was positive,
then by the Lebesgue density theorem there would be some $X\in Q_e$ 
such that the relative measure of  $Q_e$ above $X\restr_n$ tends to 1 as 
$n\to\infty$.
This contradicts the fact that every time
a permanent $e$-marker placed on $X\restr_n$ acts, 
a fixed proportion (namely $1/2^e$) of the reals extending 
$X\restr_n$ will not receive an $e$-marker again, and so will not be
in $Q_e$.  Since $\cup_e Q_e$ is $\Sigma^0_3$ and has measure $0$, 
we can compute the index of a $\emptyset'$-c.e.\ set of strings
$S$ such that $\mu(S)<2^{-k-1}$ and every real in $\cup_e Q_e$ has a prefix in 
$S$.
If we set $W=V\cup S$ then $\mu(W)< 2^{-k}$ and, for every real
that does not have a prefix in $W$, outcome (1) occurs. 

Now suppose that outcome (1) occurs for $X$.  
This means that, for each $e>k+1$ there is some longest $\tau \subset X$ on 
which a permanent $e$ marker 
is placed. There are two possibilities to consider. The first possibility is that 
$\tau\ast 0^e \subset X$ and the 
permanent marker placed on $\tau$ acts. Then $R_e$ is satisfied with respect to 
$X$, and we $\Phi$-map $
\tau\ast 0^e$ to a proper extension of $\Phi^{\tau}$. The second possibility is that 
the permanent marker on $
\tau$ does not act. Then there are no proper extensions of $\Phi^{\tau}$ in 
$W_e$. At the stage $s+1$ after 
placing the marker on $\tau$ we enumerate an axiom 
$\langle \tau',\Phi^{\tau}[s]\ast 0\rangle$ for some 
$\tau'\subset X$. 
Thus, in either case $R_e$ is satisfied with respect to $X$, and we may also 
conclude that $\Phi^X$ is total. 
\end{proof}
\noindent
Theorem \ref{th:kur2ranb1gen} says that 2-randomness is sufficient
to guarantee bounding a 1-generic.
Throughout this paper we will be concerned in establishing
optimal results, i.e.\ the `weakest' level of randomness or genericity 
that is sufficient to guarantee some property. In this case, it is not 
difficult to deal with weak 2-randomness. 
\begin{prop}
There is a weakly 2-random degree 
which does not bound any 1-generic degrees.
\end{prop}
\begin{proof}
This is a consequence of the following 
facts: (i) hyperimmune-free 1-random degrees are 
weakly 2-random, (ii) the hyperimmune-free 
degrees are downward closed and (iii) 1-generic degrees
are not hyperimmune-free.
\end{proof}
\noindent We do not know, however, whether every Demuth random bounds a 1-
generic.

\section{All non-zero degrees bounded by 
a sufficiently random degree}\label{subse:allnzdbba}
Many degree-theoretic properties $P$ that hold for all sufficiently random 
degrees
also hold for any non-zero  degree that is bounded by a sufficiently random 
degree.
In this section we show how the type of construction discussed in Section 
\ref{subse:asrdsdf},  which proves 
that a property $P$ holds for all sufficiently
random degrees, 
can be modified to show that
 $P$ holds for all non-zero degrees which are bounded by a sufficiently
random degree. Typically, `sufficient randomness' turns out to be 2-randomness.

\subsection{Methodology}
As in Section \ref{subse:asrdsdf} we break $P$ into a countable list 
$\{R_e\}_{e\in \omega}$
of simpler requirements. Given a special functional $\Theta$ we look to show that 
$P$ is satisfied by all sets 
computed by a 2-random via $\Theta$. We have an atomic strategy 
which takes
a number $e$ and a string $\sigma$ as inputs
and satisfies $R_e$ for a certain proportion of the reals that $\Theta$-map to  
extensions
of $\sigma$, where this proportion depends on
$e$ and not on $\sigma$.
Given  $k\in\omega$
we describe how to assemble a construction (from the atomic strategies)
which produces a set of strings $W$ with $\mu(W)<2^{-k}$ 
and ensures that all requirements  are met for all reals that do not have a prefix 
in $W$. 

So, to clarify, the construction is similar to the one discussed
in Section \ref{subse:asrdsdf}, only this time the $e$-markers
are to be placed on initial segments of the images
$\Theta^X$ rather than the arguments $X$ (whose initial
segments may possibly be members of $W$).
As a result of this modification, an $e$-marker that is placed
on some string $\sigma$ will strive to achieve the satisfaction of $R_e$
for a fixed proportion of the reals that $\Theta$-map to $\sigma$, rather than a 
proportion of the reals 
extending $\sigma$.

The outcomes of the construction refer to
reals $Y$ in the image space for $\Theta$,  and are the same (1), (2), (3)
as listed in  Section \ref{subse:asrdsdf}.
A density argument 
(based on Lemma \ref{le:extled}) 
suffices to show that the reals that $\Theta$-map to reals
$Y$ with infinitary 
outcome (2) form a null $\Sigma^0_3$ class.
A simple measure counting argument will show that the
reals $X$ for which $\Theta^X$ is total  and has outcome (3), are contained in
an open set of measure at most $2^{-k-1}$. This way a
set of strings $W$ of measure $<2^{-k}$ 
can be produced such that for every real $X$
without a prefix in $W$, if $\Theta^X$ is total then
it has outcome $(1)$ and therefore satisfies $P$.

We give some details concerning the standard features of such a construction
and its verification. Let us recall what took place in the proof of Theorem 
\ref{th:kur2ranb1gen}, since this 
serves as useful example. When an $e$-marker was placed on a string $\tau$, 
what we did in effect was to 
reserve a proportion $2^{-e}$ of the total measure above $\tau$. For the strings 
extending $\tau \ast 0^e$ we 
stopped enumerating axioms for $\Phi$, and we waited for a chance to satisfy 
the genericity requirement 
directly for these strings. This proportion $2^{-e}$ then played two vital roles: 
\begin{enumerate}
\item[(a)] We were able to consider the prefix-free set of strings on which 
permanent $e$-markers are placed 
but do not act, and were able to conclude that the measure permanently 
reserved by these markers is at 
most $2^{-e}$. 
\item[(b)] We were able to conclude that, when an $e$-marker placed on $\tau$ 
acts, it permanently satisfies 
the corresponding requirement for a proportion $2^{-e}$ of the total measure 
above $\tau$, so that the 
Lebesgue density theorem can be applied to show that the set of reals for which 
outcome (2) occurs is of 
measure 0.   
\end{enumerate}  
Now we look to achieve something very similar. We want conditions very similar 
to (a) and (b) to hold, but 
now, rather than considering proportions of the measure above the string on 
which a marker is placed, we 
must consider proportions of the measure that $\Theta$-maps there. The first 
important point to note is that 
we do not actually require the proportions involved in (a) and (b) to be the same. 
If we have that some 
modified version of condition (a) applies, where the proportion involved is 
$2^{-e}$, then we shall be happy 
if condition (b) applies for a smaller proportion---so long as this proportion 
depends only on $e$ and not on 
$\sigma$ we shall be able to apply Lemma \ref{le:extled} as desired.

We proceed as follows. Let us write $\pi(\sigma)$ instead of 
$\pi(\Theta, \sigma)$, and 
let $\sigma\mapsto q_{\sigma}$ be 
a computable map from strings to numbers such that:
\begin{equation}\label{eq:asumofps}
\sum_{\sigma} 2^{-q_{\sigma}}< 2^{-k-3},\ \ \ \
 \textrm{where $\sigma$ ranges over all strings.}
\end{equation}
When an $e$-marker is placed on  $\sigma$, it is given a corresponding 
parameter
 $m_{\sigma}$, which is chosen to be {\em large}. It then places 
 {\em submarkers} on all extensions of $\sigma$ of length $m_{\sigma}$.  
  The atomic strategy for the satisfaction of $R_e$ that we assume given,
 will be played individually by these submarkers.
 Each $e$-marker works with an approximation $\pi^{\ast}(\sigma)$ to
 $\pi(\sigma)$ which is initially the current value
  $\pi(\sigma)$ at the stage when the marker is placed, and is
 updated when necessary,  so as to maintain the condition that 
(\ref{eq:aaproxofpi})
 holds at stages $s$ where the value of $\pi^{\ast}(\sigma)$ is used
 by the construction:
 \begin{equation}\label{eq:aaproxofpi}
 \pi(\sigma)[s]<2\pi^{\ast}(\sigma)[s].
 \end{equation}
 Each update causes an injury of the $e$-marker and
 causes it to remove its previous submarkers (and all other markers and 
submarkers placed on proper 
extensions of $\sigma$) and redefine $m_{\sigma}$.
Clearly each marker can only be injured finitely many times in this way.
This injury is the reason that the atomic strategy is implemented
by the submarkers, rather than by the
marker itself.

An $e$-marker that sits on a string $\sigma$ is initially
{\em inactive}. An inactive marker may only be
activated by the construction at a stage $s_0$ 
if it has found a \emph{suitable} set of strings $P_{\sigma}(\sigma')$ above each 
string $\sigma'$ on which it 
has placed a submarker. We then let $P_{\sigma}$ be the union of all the various 
$P_{\sigma}(\sigma')$, as 
$\sigma'$ ranges over the strings on which it has placed submarkers.  Here 
\emph{suitable} means that the 
strings in $P_{\sigma}(\sigma')$ are all those extending $\sigma'$ of some length 
$\ell_{\sigma}>m_{\sigma}
$ and furthermore, for $s=s_0$: 
\begin{equation}\label{eq:aactivcon}
\pi(P_{\sigma})[s] \geq 2^{-k-2}\cdot \pi^{\ast}(\sigma)[s]\ \ 
\parbox{5.7cm}{and\ \ $\forall \rho\in P_{\sigma}(\sigma')[\pi(\rho)[s]<
2^{-q_{\sigma'}}]$.}
\end{equation}

\noindent Once a marker  becomes active 
it remains so until  injured or removed.

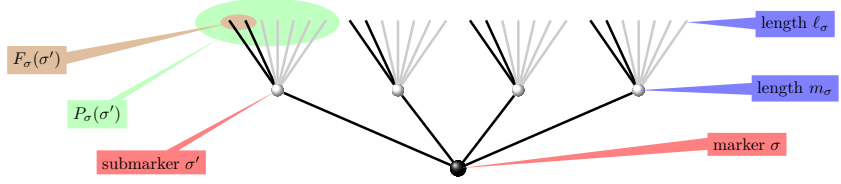
\begin{figure}[htbp]
  \begin{center}
\scalebox{0.8}{
\begin{tikzpicture}[level distance=13mm, remember picture, 
note/.style={rectangle callout, fill=#1}]
\fill [color=green!25]	(-3.17,2.43) ellipse (1.2 and .4);
\fill [color=brown!50]	(-3.65,2.43) ellipse (.3 and .13);
\tikzstyle{level 1}=[sibling distance=20mm]
\tikzstyle{level 2}=[sibling distance=3mm]
\node[ball color=black,circle, scale=0.8] {} [grow'=up]
child[solid, very thick] {node[ball color=white, circle, scale=0.6] (a) {}
child[solid, very thick] {node {}}
child[solid, very thick] {node {}}
child[solid, gray!40, very thick] {node {}}
child[solid, gray!40, very thick] {node {}}
child[solid, gray!40, very thick] {node {}}
child[solid, gray!40, very thick] {node {}}
child[solid, gray!40, very thick] {node {}}
}
child[solid, very thick] {node[ball color=white,circle, scale=0.6]{}
child[solid, very thick] {node {}}
child[solid, very thick] {node {}}
child[solid, gray!40, very thick] {node {}}
child[solid, gray!40, very thick] {node {}}
child[solid, gray!40, very thick] {node {}}
child[solid, gray!40, very thick] {node {}}
child[solid, gray!40, very thick] {node {}}
}
child[solid, very thick] {node[ball color=white,circle, scale=0.6]{}
child[solid, very thick] {node {}}
child[solid, very thick] {node {}}
child[solid, gray!40, very thick] {node {}}
child[solid, gray!40, very thick] {node {}}
child[solid, gray!40, very thick] {node {}}
child[solid, gray!40, very thick] {node {}}
child[solid, gray!40, very thick] {node {}}
}
child[solid, very thick] {node[ball color=white,circle, scale=0.6] (b) {}
child[solid, very thick] {node {}}
child[solid, very thick] {node {}}
child[solid, gray!40, very thick] {node {}}
child[solid, gray!40, very thick] {node {}}
child[solid, gray!40, very thick] {node {}}
child[solid, gray!40, very thick] {node {}}
child[solid, gray!40, very thick] {node (c) {}}
};
\node [note=red!50, , scale=0.8, callout absolute pointer= (a)] at 
(-5.1,0.1) {submarker $\sigma'$};
\node [note=red!50, , scale=0.8, callout absolute pointer={(0,0)}] at 
(4.8,0.4) {marker $\sigma$};
\node [note=blue!50, , scale=0.8, callout absolute pointer=(b)] at 
(5.6,1.33) {length $m_{\sigma}$};
\node [note=blue!50, , scale=0.8, callout absolute pointer={(3.8,2.43)}] at 
(5.6,2.4) {length $\ell_{\sigma}$};
\node [note=green!25, , scale=0.8, callout absolute pointer={(-3.9,2.2)}] at 
(-6,0.9) {$P_{\sigma}(\sigma')$};
\node [note=brown!50, , scale=0.8, callout absolute pointer={(-3.8,2.42)}] at 
(-7,1.8) {$F_{\sigma}(\sigma')$};
\end{tikzpicture}
}     
    \caption{A marker and its submarkers}
    \label{fig:piecewindtr}
  \end{center}
\end{figure}

Let us consider first what it means if a permanent marker never becomes active. 
Proposition 
\ref{prop:psigotz} ensures that for all sufficiently large potential values of $
\ell_{\sigma}$ the second 
inequality of (\ref{eq:aactivcon}) will eventually always hold. Since the set of 
strings on which we place 
permanent markers which do not become active will be a prefix-free set, Lemma 
\ref{le:cconran} then tells us 
that we can cover the set of all $X$ such that $\Theta^X$ is total and extends a 
string in this prefix-free set, 
with an open set of measure $< 2^{-k-2}$.  
   
   So now suppose that the marker becomes active at some stage $s_0$. 
The second condition of (\ref{eq:aactivcon})
allows us to consider a subset
$F_{\sigma}(\sigma')$ of each $P_{\sigma}(\sigma')$ such
that for $s=s_0$:
\begin{equation}\label{eq:aapproxcondmark} 
0\leq \pi(F_{\sigma}(\sigma'))[s] 
-2^{-e}\cdot \pi(P_{\sigma}(\sigma'))[s]\ <\   2^{-q_{\sigma'}}.
\end{equation}
 In other words,
the measure mapping to $F_{\sigma}(\sigma')$ is a good approximation to a
$2^{-e}$ slice of the measure mapping to $P_{\sigma}(\sigma')$. This 
immediately gives us, for $s=s_0$: 
\begin{equation}\label{eq:aapddxxmark} 
\pi(F_{\sigma}(\sigma'))[s]  < 2^{-e}\cdot \pi(\sigma')[s]+
 2^{-q_{\sigma'}}. 
\end{equation}

So (\ref{eq:aapddxxmark}) gives us a modified version of condition (a) which 
holds at stage $s_0$, since  
the submarker on $\sigma'$ will try to satisfy its requirement
 directly on the reals that $\Theta$-map to extensions of the strings in
 $F_{\sigma}(\sigma')$ by reserving this measure. In fact, it does just a little bit 
better than this, since the 
requirement only requires any conditions to be satisfied in the case that 
$\Theta^X$ is total. Take the union 
of all the  $F_{\sigma}(\sigma')$ as $\sigma'$ ranges over the strings on which 
submarkers are placed by 
the marker on $\sigma$, and then replace each string in $F_{\sigma}(\sigma')$ 
with the shortest initial 
segment of it which is long enough to be incompatible with all strings in 
$P_{\sigma}(\sigma')-F_{\sigma}
(\sigma')$. Call this set $D_{\sigma}$. If the marker placed on $\sigma$ is 
permanent, then for any $X$ such 
that $\Theta^X$ extends a string in $D_{\sigma}$, we shall not have to place 
further $e$-markers on initial 
segments of $\Theta^X$. It is therefore the strings which $\Theta$-map to 
extensions of strings in this set 
$D_{\sigma}$ with which we have to work to get our modified version of condition 
(b).    By the first inequality 
of (\ref{eq:aactivcon})
and the first inequality of (\ref{eq:aapproxcondmark}), 
we get that for $s=s_0$:
\begin{equation}\label{eq:arifinostchrk} 
2^{-k-2-e}\cdot \pi^{\ast}(\sigma)[s]
\leq \pi(D_{\sigma})[s].
\end{equation}
It follows from \ref{eq:aaproxofpi} in other words, that the measure of the reals 
which $\Theta$-map to 
extensions of strings in $D_{\sigma}$ is more than
a certain fixed proportion of $\pi(\sigma)$. For $s=s_0$ we have our modified 
version of condition (b): 

\begin{equation}\label{eq:condb} 
2^{-k-3-e}\cdot \pi(\sigma)[s]
\leq \pi(D_{\sigma})[s].
\end{equation}

Now what we have to do is to maintain (\ref{eq:aapddxxmark}) and 
(\ref{eq:condb}) at stages $s>s_0$. 
Actually, maintaining (\ref{eq:condb}) does not initially seem very problematic. 
While $\pi(D_{\sigma})[s]$ 
may increase as $s$ increases, (\ref{eq:aaproxofpi}) guarantees that
 $\pi(\sigma)[s]$ will not increase by 
any problematic amount---or rather that if it does, then this will constitute one of 
only finitely many injuries to 
the marker on $\sigma$.  Maintaining (\ref{eq:aapddxxmark}), however, requires 
us to do a little bit of work. It 
may be the case that as $s$ increases, $\pi(F_{\sigma}(\sigma'))$ increases for 
some $\sigma'$ on which a 
submarker has been placed, so that (\ref{eq:aapddxxmark}) no longer holds. In 
this case, we wish to remove 
some strings from $F_{\sigma}(\sigma')$. We can immediately do this if the 
second condition of 
(\ref{eq:aactivcon}) still holds for all $\rho\in F_{\sigma}(\sigma')$. In this case we 
can remove strings from 
$F_{\sigma}(\sigma')$ so that: 

\begin{equation}\label{eq:conda} 
2^{-e}\cdot \pi(\sigma')[s]\leq \pi(F_{\sigma}(\sigma'))[s]  < 2^{-e}\cdot \pi(\sigma')
[s]+ 2^{-q_{\sigma'}}. 
\end{equation}

This action may remove strings from $D_{\sigma}$ but it does not threaten 
satisfaction of (\ref{eq:condb}), 
since we still have that 
$\pi(F_{\sigma}(\sigma'))[s]\geq 2^{-e}\cdot \pi(\sigma')[s] 
\geq 2^{-e}\cdot \pi(\sigma')[s_0]$. 
We still have to deal, however, with the case that  the second 
condition of 
(\ref{eq:aactivcon}) no longer holds for all $\rho\in F_{\sigma}(\sigma')$. In this 
case, we simply choose $\ell$ 
to be large, and replace each string   $\rho\in F_{\sigma}(\sigma')$ with all 
extensions of $\rho$ of length $
\ell$, to form a new $F_{\sigma}(\sigma')$. This does not threaten satisfaction of 
(\ref{eq:condb}) because it 
does not change $D_{\sigma}$. Moreover, Proposition \ref{prop:psigotz} ensures 
that we will only have to 
redefine $F_{\sigma}(\sigma')$ in this way finitely many times.

These considerations allow
for an argument along the lines of Section  \ref{subse:asrdsdf}.
The basic features of the methodology, such as the  measure counting
which deals with  outcome (3) and the density argument which deals with 
outcome (2),
remain essentially the same.
In constructions of this form, the submarkers are primarily 
responsible for ensuring that the requirements are met. It is the submarkers that 
can act. The markers themselves can only change between being inactive and active.

\subsection{Example: downward density for 1-generic degrees}
In this section we prove Theorem \ref{th:2ranb1gen}
which says that every non-zero degree
 that is bounded by a 2-random
degree $\mathbf{a}$ bounds a 1-generic degree. This is a strengthening of
a result from \cite{Kurtz:81} 
(also discussed in \cite[Section 8.21]{rodenisbook}),
which asserted that the 1-generic degrees are downward dense in
almost all degrees (i.e.\ the class of degrees $\mathbf{a}$ 
with the above property
has measure 1). 

\begin{thm}\label{th:2ranb1gen}
Every non-zero degree that is bounded by a 2-random degree bounds a 1-
generic degree.
\end{thm}
\begin{proof}
Let $\{ W_e\}_{e\in \omega}$ be an effective enumeration of all
c.e.\ sets of strings and
suppose that $B$ is 2-random and computes
an incomputable set $A$ via the
Turing reduction $\Theta$. 
By Lemma \ref{coro:replfunat} we may assume that
$\Theta$ is special.
It suffices to define a computable procedure 
which takes as input $k\in\omega$ 
and returns the index of a $\emptyset'$-c.e.\ set 
of strings $W$ with $\mu(W)<2^{-k}$ and a functional $\Phi$
such that, if $\Theta^X=Y$ and
$X$ does not have a prefix in $W$, 
then $\Phi^Y$ is total and for all $e$: 
\[
R_e: \ \exists n\ \big[\Phi^Y\restr_n\in W_e\ \vee\ \forall \eta\in W_e,\ \Phi^Y
\restr_n \not\subseteq \eta\big].
\]
We follow the methodology and notation of
Section \ref{subse:allnzdbba}.

\ \paragraph{{\bf Construction}}
At Stage 0 place a k+4-marker on the
empty string.

At stage $s+1\in2\omega^{[e]}$, if  $e>k+3$  then
for each $e$-marker that sits on a string $\sigma$, proceed according to the first 
case below that applies.

\begin{enumerate}
\item If (\ref{eq:aaproxofpi}) does not hold, let
$\pi^{\ast}(\sigma)=\pi(\sigma)[s]$, declare that
the $e$-marker on $\sigma$ is {\em injured} and
is inactive.  Remove any markers and submarkers that sit on
proper extensions of $\sigma$.  Let $m_{\sigma}$
be large and place a submarker on each extension of $\sigma$ of length 
$m_{\sigma}$.

\item Otherwise, if the marker is inactive and 
(\ref{eq:aactivcon}) holds for some set of strings
$P_{\sigma}(\sigma')$ for each submarker $\sigma'$, 
declare that the marker is {\em active},
and define 
$F_{\sigma}(\sigma')$ for each submarker  $\sigma'$  to be a subset of 
$P_{\sigma}(\sigma')$
such that (\ref{eq:aapproxcondmark}) holds.
Moreover for each submarker $\sigma'$  and for
each extension
$\rho$ of $\sigma'$ in $P_{\sigma}(\sigma')-F_{\sigma}(\sigma')$,
define
$\Phi^{\rho}$ to be $\cup_{\rho'\subset\rho} \Phi^{\rho'}$ concatenated with $ 0$.
\item Otherwise, for each submarker $\sigma'$ which has not acted, 
such that there is an extension $\eta$ of 
  $\Phi^{\sigma'}[s]$ in $W_e[s]$, define
$\Phi^{\rho}$ to be the least such $\eta$ for all $\rho\in F_{\sigma}(\sigma')$.  In 
this case, remove
all markers and submarkers that sit on  proper extensions of $\sigma'$
 and declare that
the submarker has acted. For each submarker $\sigma'$ which has not acted, 
such that there is no extension $\eta$ of 
  $\Phi^{\sigma'}[s]$ in $W_e[s]$ and such that (\ref{eq:aapddxxmark}) no longer 
holds, there are 
  two possibilities to consider. If the second condition of (\ref{eq:aactivcon}) still 
holds for all $\rho\in 
F_{\sigma}(\sigma')$, 
  then remove strings from $F_{\sigma}(\sigma')$ so that (\ref{eq:conda}) holds. 
  If $\rho$ is removed from $F_{\sigma}(\sigma')$ then
  define $\Phi^{\rho}$ to be $\cup_{\rho'\subset\rho} \Phi^{\rho'}$
  concatenated with 0. If the second condition of (\ref{eq:aactivcon})
  does not hold then choose $\ell$ to be large, 
  and replace each string   $\rho\in F_{\sigma}(\sigma')$ 
  with all extensions of $\rho$ of length $\ell$, to form 
a new $F_{\sigma}(\sigma')$.
\end{enumerate}

At stage $s+1\in 2\omega+1$ 
let $\ell$ be large  and
do the following for each string  $\rho$ of length $\ell$, provided
that if $\sigma$ is the longest prefix of it on which a marker is placed,
then this marker is active. Let $\sigma'$ be the string of length $m_{\sigma}$ 
which is an initial segment of  $
\rho$,
and let $e$ be the index of the marker placed on $\sigma$.
If the submarker on $\sigma'$ has not acted
then put an $(e+1)$-marker on
$\rho$, unless $\rho$  extends a string
in $F_{\sigma}(\sigma')$.
If the submarker on $\sigma'$ has acted, then 
put an $e+1$ or $e$ marker on $\rho$ depending on  whether 
it has a prefix in $F_{\sigma}(\sigma')$ or not (respectively).

\ \paragraph{{\bf Verification}}
First we show that the axioms enumerated for 
$\Phi$ are consistent. The only steps of the construction at which we enumerate 
axioms for $\Phi$ are in 
clauses (2) and (3) of the even stages. Consider first the case that (2) applies at 
stage $s$. Then, prior to this 
stage, we have not enumerated any axioms for $\Phi$ with respect to strings 
extending the submarkers 
(since whenever the marker is injured because clause (1) applies we redefine 
$m_{\sigma}$ to be large). 
The axioms enumerated at this point are therefore unproblematic. Consider next 
the case that (3) applies at 
stage $s$. For each  $\rho\in F_{\sigma}(\sigma')$ for which we enumerate an 
axiom, this string is mapped 
to an extension of $\Phi^{\sigma'}[s]$, and we have not previously enumerated 
axioms with respect to proper 
extensions of $\sigma'$ which are compatible with $\rho$.

Let $T_0$ be the set of strings $\sigma$ on which we place a permanent 
marker that is always inactive after its last injury. 
No markers are placed above inactive markers,
and upon every injury through clause (1) a marker removes all markers placed 
on proper extensions.  The 
set $T_0$ is therefore 
prefix-free.
Moreover, for each $\sigma\in T_0$ we have that  
(\ref{eq:conamasmea}) holds for $c=k+2$ and for all sufficiently large $\ell$. 
 By Lemma \ref{le:cconran} we can find an index of a $\emptyset'$-c.e.\ set of 
strings $V_0$ such that 
$\mu(V_0)<2^{-k-2}$ and, if $\Theta^X$ is 
total and has a prefix in $T_0$, then 
$X$ has a prefix in $V_0$.

For each $e>k+3$ let $T_e$ be the set of strings on which we place permanent 
submarkers which do not 
act, which are placed by permanent $e$-markers which are eventually always 
active. 
If  a permanent $e$-marker is placed on $\sigma$, which places a permanent 
submarker on 
$\sigma'$ which does not act, 
then the construction will not place $e$-markers
on extensions of  $\sigma'$.
Therefore each $T_e$ is a prefix-free set. 
Let $J_e$ be the union of all $F_{\sigma}(\sigma')$
such that $\sigma' \in T_e$ and the submarker on $\sigma'$ is placed by a 
marker on $\sigma$. 
 Since we maintain  (\ref{eq:aapddxxmark}) it follows that: 
\[
\pi(J_e) < \sum_{\sigma'\in T_e} 2^{-q_{\sigma'}} + 
\sum_{\sigma'\in T_e} 2^{-e}\cdot \pi(\sigma').
\]
Summing over all $e$ it follows that we can find an index for a 
$\emptyset'$-c.e.\ set of strings $V_1$, such 
that $\mu(V_1)<2^{-k-2}$ and any $X$ such that $\Theta^X$ extends a string in 
some $J_e$ has an 
extension in $V_1$. 

So far we have dealt with the reals for which outcome (3) occurs. Next we wish to 
show: 
\begin{equation}\label{eq:mclofcdr}
\parbox{10cm}{The class of reals $X$ such that $\Theta^X=Y$ and
for some $e$ there are infinitely many permanent $e$-markers that are placed 
on initial segments of $Y$, has measure zero.}
\end{equation}
For a contradiction, assume that $e>k+3$ and that the class of reals $X$ such 
that $\Theta^X=Y$ and
 there are infinitely many permanent $e$-markers that are placed 
on initial segments of $Y$, is of positive measure. 
Let $D_e$ be the union of all the final values $D_{\sigma}$
such that a  permanent $e$-marker is placed on  $\sigma$.
Now consider the set of $X$ such that $\Theta^X$ is total and does not extend 
any string in $D_e$. This is a 
superset of the set  of reals $X$ such that $\Theta^X=Y$ and
 there are infinitely many permanent $e$-markers that are placed 
on initial segments of $Y$.  Applying Lemma \ref{le:extled} to $\Theta$
 and  $D_e$
we conclude that there exists $X$ such that for any $\epsilon>0$, there exists a 
permanent $e$-marker 
placed on $\sigma \subset \Theta^X$ which is eventually active,  for which the 
proportion of reals $\Theta$-mapped to extensions of $\sigma$ which do not map 
to extensions of any string 
in $D_{\sigma}$, is $<\epsilon$. 
This contradicts (\ref{eq:condb}). 
Since the class of (\ref{eq:mclofcdr}) is a null $\Sigma^0_3$ class,
there is a $\emptyset'$-c.e.\ set of strings $S$ such that
$\mu(S)< 2^{-k-1}$ and every real in the class has a prefix in $S$.
Moreover, an index for $S$ can be computed from an index for the given
$\Sigma^0_3$ class. We let $W = V_0 \cup V_1 \cup S$. 

Finally then, suppose that $\Theta^X=Y$ is total, and that for every $e>k+3$ 
there is a permanent $e$ 
marker placed on some initial segment of $Y$. Let $\sigma$ be the longest initial 
segment of $Y$ on which 
a permanent $e$-marker is placed.  
 This $e$-marker will become active. Let $\sigma'$ be the initial segment of $Y$ 
on which the marker on $
\sigma$ places a permanent submarker. If $Y$ extends a string in 
$F_{\sigma}(\sigma')$ then the submarker 
acts, and in doing so properly extends $\Phi^Y$ and ensures that $R_e$ is 
satisfied with respect to $Y$. 
Otherwise $Y$ does not extend a string in $F_{\sigma}(\sigma')$. In this case  
$R_e$ is automatically 
satisfied with respect to $Y$ because there do not exist any 
extensions of $\Phi^{\sigma'}$ in $W_e$. The 
length of $\Phi^Y$ is increased the last time that $\sigma$ is declared active.
\end{proof}

\section{Bounding a minimal degree}
First of all let us consider some background. Cooper showed that all high 
degrees below $\bf{0}'$ bound 
minimal degrees, and this was extended by Jockusch \cite{Jockusch:77} who 
used the recursion theorem in 
order to show that, in fact, all degrees which are GH$_1$ bound minimal 
degrees. This was shown to be 
sharp by Lerman \cite{ML86}, who constructed a high$_2$ degree which does 
not bound any minimal 
degrees. Next let us consider what happens when we consider Baire category. 
  
\subsection{Category}
As discussed in the introduction, the degrees which do not bound minimals form 
a comeager class 
\cite{Martin:60}, and the level of genericity
that guarantees this property turns out to be 2-genericity 
\cite{Yates:76, Jockusch:80}.
 On the other hand Chong
and Downey \cite{ChoDo90} and (independently)
Kumabe \cite{Kuma:90} constructed a 1-generic degree which bounds a minimal 
degree.
As a point of interest, one can also show that there are non-zero 
hyperimmune-free degrees bounded by 1-generics \cite{AL07, Downey_arithmeticalsacks},  (as well as hyperimmune-free 
degrees that are not bounded by any 1-generic degree).

\subsection{Measure}
A sufficiently random degree does not bound minimal degrees.
This follows from a paper by Paris \cite{paris77}, where it is shown that 
the degrees with minimal predecessors
form a class of measure 0.
A substantial refinement of this result was given by Kurtz \cite{Kurtz:81}
(also see \cite[Section 7.21.4]{rodenisbook}),
who showed that for almost all degrees $\mathbf{a}$ (i.e.\ all but a set
of measure 0) if $\mathbf{0}<\mathbf{b}\leq\mathbf{a}$ 
then $\mathbf{b}$ bounds a 1-generic degree.
In other words, for almost all degrees $\mathbf{a}$ 
the class of 1-generic degrees is downward dense below $\mathbf{a}$.
Since 1-generic degrees are not minimal (by \cite{Jockusch:80})
this implies Paris' result. Both of these
arguments, however, where achieved by way of contradiction 
 and do not allow
a clear view of the level of randomness that is required.
In \cite[Section 7.21.4, Footnote 15]{rodenisbook}, for example, 
the authors note that 
the precise level of randomness which guarantees Kurtz's result was not known. 
In Section \ref{se:bound1gendeg} we answered this
question by proving that every non-zero degree bounded by a 2-random 
computes a 1-generic. 

\begin{coro}
If a degree is 2-random then it does not have minimal predecessors.
\end{coro}
\begin{proof}
This is a consequence of Theorem \ref{th:2ranb1gen},
since 1-generic degrees cannot be minimal.
\end{proof}

In the remainder of this section we show that these results are optimal.
In other words, 2-randomness cannot be replaced 
with any of the standard weaker forms of randomness.
It is not hard to show that 
{\em there is a Demuth random degree which bounds a minimal degree}.
By \cite[Theorem 3.6.25]{Ottobook} there is a Demuth random
real which is $\Delta^0_2$.
All 1-random degrees, and so all Demuth random degrees, are fixed point free.  
Ku\v{c}era's technique of 
fixed point free permitting shows that all fixed point free $\Delta^0_2$ degrees 
bound non-zero c.e. degrees. 
By \cite{Yates:70*1} every non-zero c.e.\ degree bounds a minimal degree.

In order to show that there is 
a weakly 2-random degree which bounds a minimal degree we
will use the following characterization of weak 2-randomness.
\begin{equation}\label{eq:charw2r}
\parbox{11cm}{A 1-random real is weakly 2-random iff it forms 
a minimal pair with $\mathbf{0}'$.}
\end{equation}
This characterization was proved in \cite{DNWY} and
was essentially based on a theorem by 
Hirschfeldt and Miller  on $\Sigma_3^0$ 
null classes (see
\cite[Theorem 6.2.11]{rodenisbook}  or  
\cite[Theorem 5.3.16]{Ottobook} for more 
details). 
As mentioned previously, in  \cite{Jockusch:77} 
it was shown that every 
generalized high degree  bounds a 
minimal degree.
Hence to exhibit a weakly 2-random degree 
which bounds a minimal degree it 
suffices to exhibit
a generalized high weakly 2-random degree.
Given (\ref{eq:charw2r}) it suffices to show that every 
 \pz class of positive measure
has a 
member of 
generalized high degree which 
forms a minimal pair with $\mathbf{0}'$. 
For more basis theorems of this type 
(involving \pz classes and degrees which form a minimal pair 
with $\mathbf{0}'$) we refer the reader to
\cite[Sections 2,3]{BDNGP}. 
Note that this statement, which will be 
proved as Theorem \ref{th:gjhiera}, is not 
true for 
all \pz classes with no computable paths. Indeed, 
it is well known that there is such a class for which all members are generalized 
low 
(\cite{MR1720779}).

The proof of Theorem \ref{th:gjhiera} 
uses a basic strategy for dealing
with the minimal pair requirements in 
\pz classes (as in \cite[Section 2.1]{BDNGP})
combined with the method of Ku\v{c}era \cite{MR820784}
for coding information into the jump of a random set.
A detailed presentation of the latter 
can be found in  \cite[Section 1.2]{BDNGP}.
Coding into random sets (or their jumps) is based on the
following fact from
Ku\v{c}era \cite{MR820784}.
Let $ \{ P_e \}_{e\in \omega}$ be an 
effective enumeration of all \pz classes. We say $\tau$ is $P_e$-extendible if it 
has an infinite extension in 
$P_e$. 

\begin{equation}\label{eq:kuba}
\parbox{11cm}{There exists a \pz class  $P$ 
 of positive measure and  
a computable function $g$ of two arguments such that,
for all $P$-extendible strings $\tau$ and all 
$e\in\Nat$, if $P\cap P_e\cap [\tau]\neq\emptyset$ 
there exist at least two $P\cap P_e$-extendible 
strings of length $g(|\tau|,e)$ with common prefix $\tau$.}
\end{equation}
Note that (\ref{eq:kuba}) also holds for every $\Pi^0_1$ subclass 
of $P$ in place of
$P$. Moreover,  according to 
\cite{MR820784}
the class $P$ can be assumed to contain only 1-random reals
and may be chosen to have measure
that is arbitrarily close to 1.
As a consequence, for each string $\tau$, if 
$P\cap [\tau]$ is nonempty then
it has positive measure.

Roughly speaking, constructing a random set $A$ whose
jump $A'$ has a certain computational power, involves
an oracle construction that looks like forcing with \pz classes,
but typically involves injury amongst the \pz conditions.
In particular, a sequence $\{ Q_s \}_{s\in \omega} $ of \pz classes of 1-random 
reals
is defined in stages, along with a monotone sequence 
$\{ \tau_s \}_{s\in \omega}$ 
of strings (so that ultimately we can define $A=\cup_s\tau_s$) 
but we do not always have $Q_s\supseteq Q_{s+1}$.
The coding of a certain event (which is $\Sigma^0_1$ relative to the oracle used 
to run the construction)
into $A'$ is associated with a certain class $Q_s$. Then 
the $Q_{s'}$ for $ s'>s$ are defined as subclasses of $Q_s$ and
the $\tau_{s'}$ for  $s'>s$ are extendible in $Q_s$. If and when the
aforementioned $\Sigma^0_1$ event occurs, however, the construction
defines an initial segment of $A$ in such a way as to ensure  $A\not\in Q_s$.
This action codes the event into $A'$ and may cause injury
to lower priority requirements (whose satisfaction relied on a
\pz condition that may no longer be valid).
This intuitive description may be helpful in visualising
the proof of Theorem \ref{th:gjhiera}.

\begin{thm}\label{th:gjhiera}
Given a  \pz class $P$ of positive measure  
there is  $A\in P$ which is generalized high and 
forms a minimal pair with $\emptyset'$. 
Moreover $A\leq_T\emptyset''$.
\end{thm}
\begin{proof}
The construction is a forcing argument
 with \pz classes
of positive measure, in which we allow finite injury amongst
the \pz conditions (and the requirements that these represent). 
The construction will proceed
in stages, computably in $\emptyset''$, 
defining a \pz class $Q_s$ and a string $\tau_s$
at stage $s\in\omega$. We will have 
$\tau_s\subset\tau_{s+1}$ for each $s$ and
will eventually define $A=\cup_s\tau_s$. 
However, we may have $Q_s\not\supseteq Q_{s+1}$, which indicates an injury 
that is caused
by the coding of $(A\oplus \emptyset')'$ into $A'$.
The minimal pair requirements may be expressed as follows:
\[
R_e: \textrm{If}\  \Psi_e^{\emptyset'}\ \textrm{is total and incomputable}\  
\textrm{then}\ 
 \Psi_e^{\emptyset'}\neq \Psi_e^A.
\]
Stages in $2\omega^{[e]}$ will be devoted to the satisfaction of
$R_e$. We may need to act (finitely) many times for each $R_e$
due to the injuries to requirements that may occur. 
Stages in $2\omega +1$ will be devoted to coding $(A\oplus\emptyset')'$
into $A'$. In particular, 
stages in $2\omega^{[e]} +1$ are devoted to satisfying the requirement $N_e$ 
that we
 code into $A'$ whether or not $e$
belongs to  $(A\oplus\emptyset')'$.
By \cite{MR820784}
 we may assume that the given
class $P$ is the same as the class of
(\ref{eq:kuba}), with the additional properties mentioned in the
paragraph below it. 
Let $\tau_0=\emptyset$, $Q_0=P$ and consider the function $g$
of (\ref{eq:kuba}). For the purposes of this proof we assume that if 
$n\in \omega^{[e]}$ then either 
$n+1\in \omega^{[e+1]}$ or $n+1\in \omega^{[0]}$.

\ \paragraph{{\bf Construction}}
At stage $s+1\in 2\omega^{[e]}$ let $j_s$ be an index for $Q_s$. 
Let $\rho_0$ and $\rho_1$ be, respectively, the leftmost and rightmost 
extensions of $\tau_s$ which are extendible in $Q_s$ and are  of length
$g(|\tau_s|,j_s)$.
 Check to see whether there exists $n$ such that  
 $\Psi_e^{\emptyset'}(n)\de=m$ and:
 \begin{equation}\label{eq:pedefcon}
 Q_{s}\cap [\rho_1]\cap \{X\ |\ \Psi_e^X(n)\downarrow \neq m\ 
 \vee\ \Psi_e^X(n)\un\}\neq\emptyset.
 \end{equation}
 If there is such $n$ then consider the least one,  set $Q_{s+1}$ 
 equal to the \pz class of (\ref{eq:pedefcon}) and define
 $\tau_{s+1}=\rho_1$. Otherwise, 
 let $Q_{s+1}:=Q_{s}\cap [\rho_0]$ and define $\tau_{s+1}=\rho_0$.

At stage $s+1\in 2\omega^{[e]}+1$ let $j_s$ be an index for $Q_s$. 

We consider first the case that $Q^{\ast}_e$ and $f_e$ are undefined. In this 
case proceed as follows. 
Let $Q^{\ast}_e=Q_s$
and define $f_e$ by recursion:  $f_e(0)=|\tau_s|$ and 
$f_e(k+1)=g(f_e(k),j_s)$.
Also, let $Q_{s+1}$ consist of all elements of $Q_s$ except
those that extend any string $\rho$ which satisfies the following: 
there exists $k\in \omega $ and  $\tau$ of 
length $f_e(k)$, such that $\rho$ is the leftmost extension of $\tau$  of length 
$f_e(k+1)$ which is extendible 
in $Q_s$. 
By the choice of $g$ it follows that $Q_{s+1}$ is
a non-empty \pz class.
Also let $\tau_{s+1}$ be the 
leftmost one-bit extension of $\tau_s$ which is extendible in
$Q_{s+1}$.

If $Q^{\ast}_e, f_e$ are defined at stage $s+1$, 
let $t$ be the stage at which they
were last defined (i.e.\ the greatest stage $\leq s$ such that these values were 
undefined at the beginning of 
the stage and were made defined according to the instructions for that stage). If 
$N_e$ {\em acted} 
after stage $t$ or 
$\Psi_e^{\tau_s\oplus\emptyset'}[s]\un$, 
then let $Q_{s+1}=Q_s$ 
and let $\tau_{s+1}$ be the 
leftmost one-bit extension of $\tau_s$ which is extendible in
$Q_s$.
On the other hand, if  
$\Psi_e^{\tau_s\oplus\emptyset'}[s]\de$, 
then let $\rho$ be the least $Q_s$-extendible extension
of $\tau_s$ of length in $f_e(\omega)$ and define
$\tau_{s+1}$ to be the leftmost extension of $\rho$
of length $f_e(|\rho|)$
which is extendible in $Q^{\ast}_e$. In the latter case define  
$Q_{s+1}=Q^{\ast}_e$, declare  that $N_e$ has 
{\em acted} at this stage
 and make $Q^{\ast}_j, f_j$ undefined for all $j>e$.
Note that when determining the value of 
$\Psi_e^{\tau_s\oplus\emptyset'}[s]$, the construction uses the true
initial segment of $\emptyset'$ of length $s$, and not the result of enumerating
$\emptyset'$ for $s$ steps.

 \ \paragraph{{\bf Verification}}
Let $A=\cup_s\tau_s$ and note that $A\in P$. 
First, we show  by induction on $e$ that 
each $N_e$ acts finitely often (with $Q^{\ast}_e$ and $ f_e$  eventually being 
permanently defined).
Suppose that this holds for all $N_j$, $j<e$.
At the first stage $s_0$ in $2\omega^{[e]}+1$ 
after the last action of some $N_j$, $j<e$
the construction will define $Q^{\ast}_e$ and $ f_e$. By the choice
of $s_0$ it follows that these values will never subsequently be made undefined.
 Therefore after
stage $s_0$ requirement $N_e$ can act at most once.
This concludes the induction step.

We show next that $A$ satisfies all $R_e, e\in\omega$.
Pick $e\in\omega$ and consider the least stage $s+1$ in 
$2\omega^{[e]}$ which is greater than all the stages at which 
some $N_j$ acts for $ j<e$. 
Then $A\in Q_{s+1}$ because we have $Q^{\ast}_j\subseteq Q_{s+1}$
for all $j$ such that $N_j$ acts in later stages.
If $Q_{s+1}$ is defined according to 
(\ref{eq:pedefcon}) then clearly  
$ \Psi_e^{\emptyset'}(n)\neq \Psi_e^A(n)$.
If, on the other hand, we define 
$Q_{s+1}:=Q_{s}\cap [\rho_0]$, this means that
either $\Psi_e^{\emptyset'}$ is partial or  
$\Psi_e^X$ is total for all $X\in Q_s\cap [\rho_1]$
and agrees with  $\Psi_e^{\emptyset'}$. The latter
condition implies that 
 $\Psi_e^{\emptyset'}$ is computable. In either case 
 $A$ satisfies $R_e$.
 
 It remains to show that
 $(A\oplus\emptyset')' \leq_T A'$.
 First of all note that the construction is not only computable in
 $\emptyset''$ (so that $A\leq_T\emptyset''$) but also 
 $A\oplus\emptyset'$. Indeed, the only place
 where we used more than $\emptyset'$
 in order to define $\tau_{s+1}$ and $Q_{s+1}$
 was in stages $2\omega^{e}$.
In these stages, in order to decide which clause we follow
it suffices to calculate $\rho_0$ and $\rho_1$
(using $\emptyset'$) and check which of these strings
the set $A$ extends. If it extends $\rho_1$ then we
defined $Q_{s+1}$ according to (\ref{eq:pedefcon}); otherwise
we followed the second clause.

The algorithm which calculates $(A\oplus\emptyset')'$ from $A'$
 is as follows.
Given $e\in\omega$ suppose that we have used the oracle for $A'$ to calculate
$(A\oplus\emptyset')'\restr_e$ and the least stage $s_e$ after
which no $N_j, j<e$  {\em acts}. 
Let $t_e>s_e$ be the least in $2\omega^{e}+1$. 
Then by stage $t_e$ the parameters
$Q^{\ast}_e, f_e$ have reached their eventual values. Moreover, using
$A, \emptyset'$, we may play back the construction up to this stage
and calculate the final values of $Q^{\ast}_e$ and $ f_e$. Then  
$e\in (A\oplus\emptyset')'$ if and only if 
$N_e$ {\em acts},
and this happens if and only if there exists $k\in\omega$ such that
$A\restr_{f_e(k+1)}$ is the leftmost extension of
$A\restr_{f_e(k)}$ of length $f_e(k+1)$ which is extendible in $Q^{\ast}_e$.  
Once we have determined 
whether $N_e$ acts subsequent to stage $t_e$, this suffices to specify 
$s_{e+1}$. 
\end{proof}
\noindent
We can now obtain the desired result.
\begin{coro}\label{coro:wdbminim}
There is a weakly 2-random degree which bounds a minimal degree.
\end{coro}
\begin{proof}
This is a consequence of (\ref{eq:charw2r}),
combining the fact from Jockusch \cite{Jockusch:77} that
every GH$_1$ degree bounds a minimal degree, and the 
application of Theorem \ref{th:gjhiera} to 
a nonempty \pz class which consists entirely of \ml random paths. 
\end{proof}
\noindent
Note that by Theorem \ref{th:gjhiera}, the degree of
Corollary \ref{coro:wdbminim} may be chosen below $\mathbf{0}''$.

Theorem \ref{th:gjhiera}  may be seen as a dramatic strengthening of the result 
proved in \cite{LMN07}, that 
there exists a weakly 2-random set which is not generalized low. It  also gives a 
rather simple 
positive answer to
\cite[Problem 3.6.9]{Ottobook} which asks whether all weakly
2-random sets are array computable, since  array computable
sets $A$ are generalized low$_2$. 
 This problem was first solved in
\cite[Section 5]{BDNGP} where a much stronger result
was shown using a  different but more complicated argument.
It was shown there that for every function $f$ there exists
a function $g$ which is computable in a weakly 2-random
set and which is not dominated by $f$.

\section{Minimal covers}
First of all we consider some background. The most well known theorem here is 
the result of Jockusch that 
there exists a cone of minimal covers \cite{CJ73}. This follows from the fact that 
the corresponding Gale-
Stewart game is determined. By considering a pointed tree such that every path 
through the tree is a play of 
the game according to the winning strategy, we conclude that either there is a 
cone of minimal covers, or 
else a cone of degrees which are not minimal covers. Clearly the latter is 
impossible. Next let us consider 
what happens when we consider Baire category. 
 
\subsection{Category}
The degrees that are minimal covers form a comeager set, so a
sufficiently generic degree is a minimal cover of some other degree.
In fact, Kumabe \cite{Kuma:93} 
showed that for each $n>1$, 
every $n$-generic is a minimal cover of an $n$-generic. The question left open 
here, is as to whether or not 
this result is sharp:

\begin{question} Is every 1-generic degree a minimal cover?
\end{question}

\noindent At the time of writing it seems likely that Durrant and Lewis are able to 
answer this question in the 
negative.

\subsection{Measure}
Not very much is known as regards the measure theoretic case here. The basic 
question remains: 
\begin{question} \label{mincovq}
What is the measure of the degrees which are a minimal cover? 
\end{question}
\noindent 
By \cite{Kurtz:81, Kautz:91} (also see \cite[Section 8.21.3]{rodenisbook})
 every 2-random degree is c.e.\ relative to some 
degree strictly below it.
Hence we may deduce that every 2-random degree bounds a minimal cover. 
This follows by 
relativizing the proof from 
\cite{Yates:70*1} that every non-zero c.e.\ degree bounds a minimal degree. 
Thus, if we are to believe the 
heuristic principle, that properties satisfied by all highly random degrees are likely 
to hold for all non-zero 
degrees below a  highly random, then we would expect the answer to Question 
\ref{mincovq} to be 1.

\section{Strong minimal covers and the cupping property}
A degree $\mathbf{a}$ is a strong minimal cover of another degree
$\mathbf{b}<\mathbf{a}$ if for all degrees $\mathbf{x}<\mathbf{a}$
we have $\mathbf{x}\leq \mathbf{b}$.
Notice that a strong minimal cover is not the join of two lesser degrees.
All the known examples of degrees that fail to have
a strong minimal cover satisfy the  {\em cupping property}.
A degree $\mathbf{a}$ is said to have this property if for all
$\mathbf{c}>\mathbf{a}$ there exists $\mathbf{b}<\mathbf{c}$
such that $\mathbf{a}\vee \mathbf{b}=\mathbf{c}$.
Clearly, a degree which has a strong minimal cover fails to satisfy the
cupping property. However it is not known if the converse holds.

\subsection{Category}
It is important to distinguish between the degrees that \emph{are} a strong 
minimal cover and the degrees 
which \emph{have} a strong minimal cover. The strong minimal covers form a 
meager class: if $A\oplus B$ is 
1-generic then the Turing degrees of $A$ and $B$ are
strictly less than the degree of $A\oplus B$ . Hence strong minimal covers
are not 1-generic.
On the other hand, 
the degrees which satisfy the cupping property form a comeager class, and so 
 the degrees which have a strong minimal cover also form a meager class.
In fact, Jockusch \cite[Section 6]{Jockusch:80} showed that
every $2$-generic degree has the cupping property and thus fails to
have a strong minimal cover. This can easily be extended to the weakly 2-
generics, by showing that all 
weakly 2-generics are a.n.r.\footnote{Recall that $A$ is array non-recursive 
(a.n.r.) if, for every 
$f\leq_{wtt} \emptyset'$ there exists 
$g\leq_T A$ which is not dominated by $f$.}, since it was 
shown in \cite{DJS96} that 
all a.n.r.\ degrees satisfy the cupping property. In order to show that every weakly 
2-generic set $A$ is a.n.r., 
consider the function $g_A$ which specifies the number of consecutive  0s in the 
obvious way, so that if 
\[ A=11001111000011\cdots \] 
then $g_A(0)=2$ and $g_A(1)=4$, for example. Given $f\leq_T \emptyset'$ (we 
do not require $f\leq_{wtt} 
\emptyset'$), let $h\leq_T \emptyset'$ be a function which on input $\sigma$ 
outputs $\tau \supset \sigma$ 
with $g_B(|\sigma|)>f(|\sigma|)$ for all $B\supset \tau$. For every $l$, let 
$V_l=\{ h(\sigma) :\ |\sigma|>l \}$. 
Each $V_l$ is dense, so $A$ must have an initial segment in each $V_l$. Thus 
$g_A$ is not dominated by 
$f$. 
 
On the other hand, Kumabe \cite{Kuma:00}
constructed a 1-generic degree with a strong
minimal cover.

\subsection{Measure}
The strong minimal covers form a null class. Indeed, 
if $A\oplus B$ is 1-random then the Turing degrees of $A$ and $B$ are
strictly less than the degree of $A\oplus B$ . Hence strong minimal covers
are not 1-random. We shall show in Section \ref{sec:joinprop} that, in fact, every 
non-zero degree bounded 
by a 2-random satisfies the join property, and this suffices to show that no degree 
bounded by a 2-random is 
a strong minimal cover. 
On the other hand, the measure of the degrees which \emph{have} a strong 
minimal cover is 1.
Barmpalias and Lewis showed in \cite{BL2011} that every 2-random degree
has a strong minimal cover, and so fails to satisfy the cupping property.
In the same paper we pointed out that this result fails if 2-randomness
is replaced with weak 2-randomness.

\begin{thm}
Every degree that is bounded by a 2-random degree has a 
strong minimal cover. Hence no such degree has the cupping property.
\end{thm}
\begin{proof}
We assume that the reader is familiar with the proof described in  \cite{BL2011} 
and describe only the 
modifications required to give the stronger result. Recall that 
$T\subseteq 2^{<\omega}$ is {\em perfect} if it 
is non-empty and, for all $\tau \in T$, 
there exist incompatible strings $\tau_0,\tau_1$ which extend $\tau$ 
and belong to $T$. Our main task in the proof of  \cite{BL2011} is to show that 
there exists 
$f\leq_T \emptyset'$ such that, for any 
$j,n\in \omega$, $f(j,n)=e$ which satisfies:
\begin{itemize}
 \item $\mu(W_e^{\emptyset'})<2^{-n}$;
  \item if $X\notin \llbracket W_e^{\emptyset'}\rrbracket$ 
  and $X$ computes $T$ which is perfect  
via $\Psi_j$, then it computes 
a perfect pointed  $T'\subseteq T$.
  \end{itemize}
  
\noindent Here $W_e^{\emptyset'}$ is the $e$th set of strings which is c.e.\ 
relative to $\emptyset'$.  In order 
to specify $W_e^{\emptyset'}$ we consider a computable construction which 
enumerates axioms for two 
functionals $\Phi$ and $\Xi$. The idea is that,  if 
$X\notin \llbracket W_e^{\emptyset'}\rrbracket$ 
and $X$ computes $T$ which 
is perfect  via $\Psi_j$, then $\Xi^X$ will be some perfect $T'\subseteq T$ and, 
for all $Y$ which are paths 
through $ T'$, $\Phi^Y=X$.  During the course of constructing $\Phi$ and $\Xi$, 
we consider various sets 
$S$ of finite strings 
$\tau$ for which $\Psi_j^{\tau}$ is of at least a certain length. 
Then we enumerate axioms 
for $\Phi$ and $\Xi$ in such a way that, for a high proportion of the strings in $S$, 
$\Xi^{\tau}$ is an 
appropriate subtree $T' \subseteq \Psi_j^{\tau}$ and, for all $\sigma\in \Xi^{\tau}$, 
$\Phi^{\sigma}$ is an 
initial segment of $\tau$ of appropriate length. During this process it may be that 
$\tau,\tau'\in S$ and $\tau$ 
is incompatible with $\tau'$ but $\Psi_j^{\tau}=\Psi_j^{\tau'}$. In this case we 
might define $\Xi^{\tau}$ and 
$\Xi^{\tau'}$ differently. The small modification required in order to give the 
stronger result is simply to remove 
this possibility. Now the idea is that  if 
$X\notin \llbracket W_e^{\emptyset'}\rrbracket$ and $X$ 
computes $T$ which is perfect  
via $\Psi_j$, then $\Xi^T$ will be some perfect $T'\subseteq T$ and, for all $Y$ 
which are paths through 
$T'$, $\Phi^Y=T$. Now when $\Psi_j^{\tau}=\Psi_j^{\tau'}$, it is this single value 
which we must consider as 
the oracle input for $\Xi$, rather than the two values $\tau$ and $\tau'$ as 
previously. There is no longer the 
possibility of mapping to two distinct values.  This does not cause any problems, 
because now we are only 
required to ensure that, if $X$ doesn't have any initial segment in 
$W_e^{\emptyset'}$ and $\Psi_j^X=T$ is 
perfect, then for all $Y$ which are paths through $\Xi^{T}$, $\Phi^{Y}=T$, i.e. it 
only the value $T$ that $Y$ 
must compute rather than the various $X$ such that $\Psi_j^X=T$, so there is no 
need to map to two distinct 
values anyway. 
\end{proof}

\begin{coro}
Every 2-random degree forms a minimal pair with every 2-generic degree. 
\end{coro}
\begin{proof}
As mentioned previously, Jockusch showed that all 2-generics satisfy the 
cupping property. Martin 
\cite{Martin:60} showed that, if $\bf{a}$ is n-generic and 
$\boldsymbol{0}<\boldsymbol{b}<\boldsymbol{a}$ 
then $\boldsymbol{b}$ bounds an  $n$-generic. Since the degrees which satisfy 
the cupping property are 
upward closed, it follows that all non-zero degrees below a 2-generic satisfy the 
cupping property and are 
therefore not bounded by a 2-random. 
\end{proof}

%
 
\section{The join property}\label{sec:joinprop}
A degree $\mathbf{a}$ satisfies the 
join property if for every non-zero degree
$\mathbf{b}< \mathbf{a}$ there exists $\mathbf{c}<\mathbf{a}$
such that $\mathbf{b}\vee \mathbf{c}=\mathbf{a}$. The strongest positive result 
here \cite{DGLM11} is that all 
non-GL$_2$ degrees satisfy the join property. The degrees which satisfy the join 
property, however, are not 
upward closed, and it remains open as to whether $\boldsymbol{0}'$ can be 
defined as the least degree 
such that all degrees above satisfy the join property. 

\subsection{Category}
The degrees which satisfy the join property
form a comeager class. Indeed, Jockusch \cite[Section 6]{Jockusch:80}
showed that every $2$-generic degree satisfies the join property.
He also showed that every degree that is bounded by a 2-generic degree
satisfies the join property.
In this section we show that every 1-generic degree has the join property.
The coding that we employ is based on the classic and elegant
method that was used in \cite{DBLP:journals/jsyml/PosnerR81}
for the proof that $\mathbf{0}'$ has the join property.
\begin{thm}\label{th:1genjoin}
Every 1-generic degree satisfies the join property.
\end{thm}
\begin{proof}
We suppose we are given $A$ which is 1-generic and also an incomputable set  
$B<_T A$. We may 
suppose that $B$ is not c.e., since anyway $B\oplus \bar B$ 
is not c.e.\ when $B$ is incomputable,  and is of 
the same degree as $B$. We wish to construct $C<_T A$ such that 
$A\leq_T B\oplus C$. In order to do this 
we suppose given an arbitrary set $X$ and we build $C_X$. For some $X$ we 
will have that $C_X$ is a 
partial function, but  $C_A$ will be total and will be the required joining partner for 
$B$. 

Let $\Psi$ be such that $\Psi^A=B$, and assume that this functional satisfies all 
of the conventions satisfied 
by any $\Psi_j$ as specified in Section \ref{techback}.  We also assume that, for 
any $\rho$ and any $n$, if $
\Psi^{\rho}(n)\downarrow$ then $\Psi^{\rho}(n)\in \{ 0,1 \}$. Let 
$\sigma_m=0^m1$. 
Let $\psi(X;n)$ be the use of the computation $\Psi^X(n)$ (so that if 
$\Psi^{X}(n)\uparrow$ then $\psi(X;n)\uparrow$). 
 We define a function $f_X$, which may be partial. 
For any $n$, if $\psi(X;n)\uparrow$ then let 
$f_X(n)$ be undefined, and otherwise let $\rho$ be the initial segment of $X$ of 
length $\psi(X;n)$. If there 
exists $m$ such that 
$\rho \ast \sigma_m\subset X$ then let $f_X(n)=|\rho \ast \sigma_m|$.

We consider given some fixed effective splitting search procedure which 
enumerates all unordered pairs 
$\{  \rho_0,\rho_1 \}$ such that $\rho_0$ and $\rho_1$ are $\Psi$-splitting but 
there does not exist any $
\rho_2$ such that either $\rho_2\subset \rho_0$ and $\rho_2$ and 
$\rho_1$ are $\Psi$-splitting, or  
$\rho_2\subset \rho_1$ and $\rho_2$ and 
$\rho_0$ are $\Psi$-splitting. So the 
procedure enumerates all 
pairs of strings which are $\Psi$-splitting and such that neither string can be 
replaced by a proper initial 
segment to form a new splitting. 
In order to define $C_X$, we define a sequence of strings 
$\{ \tau_{X,s} \}_{s\geq 0}$ so that $C_X=\bigcup_s \tau_{X,s}$.  
As we define the sequence  $\{ \tau_{X,s} \}_{s\geq 0}$ 
we also define  sequences  $\{ n_{X,s} \}_{s\geq 1}$ 
and $\{\rho_{X,s} \}_{s\geq 0}$. 
The sequence $\{ n_{X,s} \}_{s\geq 1}$ just keeps 
track of which bit of $\Psi^X$ we make use of at each stage of the construction. 
The sequence 
$\{\rho_{X,s} \}_{s\geq 0}$ 
records the initial segment of $X$ used by the end of stage $s$. 
This means that for all $Y
\supset \rho_{X,s}$ the construction will run in an identical way up to the end of 
stage $s$. 

The construction is required to be a little more subtle than it might initially seem. 

\ \paragraph{{\bf Construction}}
Stage 0. Define $\tau_{X,0}=\rho_{X,0}=\emptyset$. 

\noindent Stage $s+1\in \omega^{[i]}$. Search until a first pair 
$\{ \rho_0, \rho_1 \}$  is enumerated by the 
splitting search procedure such that both of $\rho_0$ and $\rho_1$ 
extend $\rho_{X,s}$  and one of these 
strings, $\rho_0$ say, is an initial segment of $X$. \footnote{It may be the case 
that no such pair is 
enumerated, in which case the construction simply continues this search for ever 
and $\tau_{X,s+1}$ 
remains undefined.}

For use in the verification it is also useful to enumerate a certain set $V_{X,i}$. 
Let  $\{ \rho_2, \rho_3 \}$  be 
the first pair enumerated by the splitting 
search procedure such that both of $\rho_2$ 
and $\rho_3$ extend $\rho_{X,s}$. Let $n_0$ be the least such that 
$\Psi^{\rho_2}(n_0)\downarrow \neq \Psi^{\rho_3}(n_0)\downarrow$, 
let $d\in \{ 2,3 \}$ be such that $\Psi^{\rho_d}(n_0)=1$ and 
enumerate $\rho_d$ into  $V_{X,i}$. 

Now we pay attention again to the pair $\{ \rho_0, \rho_1 \}$.  Let $n_1$ be the 
least such that $\Psi^{\rho_0}
(n_1)\downarrow \neq \Psi^{\rho_1}(n_1)\downarrow$. 
The remaining  instructions for the stage are divided into steps $t\geq 0$. 

Step $t$. Check to see whether there exists a least $n$ with 
$n_1 \leq n \leq n_1+t$ such that either: 
\footnote{Note that when we write ``$\Psi^X(n)$" in case (a) and case (b)  this 
denotes its final value; if $
\Psi^X(n)\uparrow $ or $f_X(n)\uparrow$ for any $n$ with  
$n_1 \leq n \leq n_1+t$ then the construction with 
respect to $X$ does not terminate at stage $s+1$ and we perform no further 
instructions.}

\begin{enumerate}
\item[(a)] $\Psi^X(n)=1$ and there does not exist any 
$\Psi_i$-splitting above $\tau_{X,s} \ast \sigma_n$ with 
the strings of length $\leq f_X(n)$, or;
\item[(b)] $\Psi^X(n)=0$ and there  
does exist a $\Psi_i$-splitting above 
$\tau_{X,s} \ast \sigma_n$ with the 
strings of length $\leq f_X(n)$.

\end{enumerate}

\noindent   If there exists no such $n$, 
then proceed to step $t+1$, otherwise let 
$n$ be the least such and 
define $n_{X,s+1}=n$.  If case (a)  applies 
for $n$, then  define $\rho_{X,s+1}$ to 
be the initial segment of $X
$ of length $f_X(n_1+t)$ and define 
$\tau_{X,s+1}=\tau_{X,s} \ast \sigma_n \ast X(s)$.  If case (b) applies for 
$n$, then let $\tau$ and $\tau'$ be the first $\Psi_i$-splitting above 
$\tau_{X,s} \ast \sigma_n$ found by some 
fixed computable search procedure. Let $n_2$ be 
the least such that 
$\Psi_i^{\tau}(n_2)\downarrow \neq \Psi_i^{\tau'}(n_2)\downarrow$ 
and let $\tau''\in \{ \tau,\tau' \}$ be such that $\Psi_i^{\tau''}(n_2)\neq X(n_2)$.  
Let $m=f_X(n_1+t)$ and define 
$\rho_{X,s+1}$ to be the initial segment of $X$ of length $m$ (note that 
$m\geq n_2$).  Define $\tau_{X,s+1}=\tau'' \ast X(s)$.  
For future reference, when case (b) occurs we also 
enumerate $\rho_{X,s+1}$ into the 
set $S_{X,i}$. This records that we have managed to 
directly diagonalize for $\Psi_i$ at this stage. Whether case (a) or 
case (b) applies, proceed 
to stage $s+2$.

\ \paragraph{{\bf Verification}} Since 
$\Psi^A$ is total and there exist infinitely 
many $n$ such that $A(n)=1$, 
it follows that $f_A$ is total. Also, since 
$\Psi^A$ is total and incomputable, for 
every initial segment $\rho$  
of $A$ there exists a pair $\{ \rho_0,\rho_1 \}$ 
enumerated by the splitting search 
procedure such that both of 
these strings extend $\rho$ and one of them 
is an initial segment of $A$. In order 
to show that $C_A$ is total, 
it therefore suffices to show that when the 
construction is run for $X=A$ there are 
only finitely many steps $t$ 
run at each stage of the construction. 
So suppose otherwise, and let $s$ be the 
least such that are an infinite 
number of steps run at stage $s+1$ of 
the construction. Let $n_1$ be as defined 
in the instructions for that 
stage. Then, for all $n\geq n_1$, if $n\in B$ then  
 there   does exist a $\Psi_i$-splitting 
 above $\tau_{X,s} \ast \sigma_n$, and if 
 $n\notin B$ then there does 
not exist a $\Psi_i$-splitting above 
$\tau_{X,s} \ast \sigma_n$. This means that 
$B$ is c.e., contrary to 
assumption. 

 Having established that $C=C_A$ is total, 
 we wish to show next that 
 $B\oplus C$ can compute the 
sequence $\{ \tau_{A,s} \}_{s\geq 0}$, 
and that therefore $A\leq_T B\oplus C$.  
Suppose inductively that 
$B\oplus C$ has already been able to 
decide $\tau_{A,s}$. Then there exists a 
unique $n$ such that $
\tau_{A,s} \ast \sigma_n \subset C$. 
This value of $n$ is $n_{A,s+1}$. By 
checking whether $n\in B$ or not, 
$B\oplus C$ can now decide whether 
case (a) or case (b) applied for $n$ at the 
step when $\tau_{A,s+1}$ 
was defined, and this is sufficient 
information to be able to determine 
$\tau_{A,s+1}$.  
 
  We are therefore left to prove that 
  $C<_T A$. Fix $i\in \omega$.  
  Let $S=\bigcup_X  S_{X,i}$. If there is 
some initial segment of $A$ in $S$ then it 
is clear that $A\neq \Psi_i^{C}$,  so 
suppose otherwise. Next 
suppose there exists a stage $s+1$ such that:

\begin{enumerate}
\item  $s+1\in \omega^{[i]}$;
\item  For the step $t$ at which stage 
$s+1$ terminates, case (a) applies for 
$n_{A,s+1}$.
\item  Putting $n=n_{A,s+1}$, there does not exist any 
$\Psi_i$-splitting above $\tau_{A,s} \ast \sigma_n$. 
\end{enumerate}

\noindent In this case it is clear that 
$\Psi_i^{C}$ is either partial or computable, 
so $A\neq \Psi_i^{C}$.

Finally, suppose that neither of these 
two cases occur.  This means that as we 
run the construction for $X=A
$, for every $s+1\in \omega^{[i]}$  and for $n=n_{A,s+1}$, 
case (a) applies for $n$ at the step of stage $s+1$ 
at which we define $\tau_{A,s+1}$, 
but \emph{actually} there does exist 
some $\Psi_i$-splitting above $\tau_{A,s}\ast \sigma_n$. 
Now we look to derive a contradiction, by showing that 
for each $\rho \subset A$ 
there are strings in $S$ extending $\rho$. 

Let $V= \bigcup_X V_{X,i}$. 
Since $A$ is 1-generic and $V$  is c.e.\ and all initial 
segments of $A$ have 
extensions in $V$, it follows 
that there are infinitely many strings in $V$ which 
are initial segments of $A$. 
Now we have to establish exactly what this means. 
Suppose $\rho\in V$ and $\rho \subset A$. Then there 
exists some $X$ such that $\rho$ is 
enumerated into $V_{X,i}$ during stage $s+1$ 
of the construction for $X$. 
Since $\rho \subset A$ it must be 
the case that $\rho_{X,s}\subset A$. This 
means that, up until the end of 
stage $s$ the constructions for $X$ and $A$ are 
identical and $\rho_{X,s}=\rho_{A,s}$.  Therefore $\rho$ is 
also enumerated into $V_{A,i}$ at stage 
$s+1$ of the construction for $A$ and 
the pairs $\{ \rho_0, \rho_1\}$ 
and $\{ \rho_2,\rho_3 \}$ as specified in 
the instructions for that stage are 
identical. Without loss of generality, 
suppose that $\rho_0=\rho_2 \subset A$ 
and let $n_0=n_1$ be as defined in the 
instructions of the 
construction for $A$ at that stage. 
Then $\Psi^{\rho_0}(n_0)=1$. There are now 
two possibilities to consider.

 First, suppose that $n_{A,s+1}=n_0$.  
 Then case (a) applies for $n_0$ at step 0 
when we define $\tau_{A,s
+1}$ but actually there does exist a $\Psi_i$-splitting  above 
$\tau_{A,s}\ast \sigma_{n_0}$. Let $r$ be 
greater than the length of the 
strings in the first such splitting. Then 
$\rho_1 \ast \sigma_r$ is a string in $S$ 
extending $\rho_{A,s}$. This follows 
because, when we run the construction for 
any $Y\supset \rho_1 \ast 
\sigma_r$, it will be identical to the construction 
for $A$ up until the end of stage 
$s$. Then $\{\rho_0,\rho_1 
\}$ will be the first pair enumerated by the 
splitting search procedure such that 
both of $\rho_0$ and $
\rho_1$ extend $\rho_{Y,s}$  and one 
of these strings is an initial segment of
 $Y$. Now {\em here} is the 
crucial point: at step $t=0$ in stage $s+1$ 
of the construction for $Y$ we find that 
$\Psi^Y(n_0)\downarrow =0$ 
and that there does exist a $\Psi_i$-splitting above 
$\tau_{Y,s}\ast \sigma_{n_0}$ with the strings of 
length less than $f_{Y}(n_0)$. 
 
 Next suppose that $n_{A,s+1}\neq n_0$. 
 Since $\Psi^A(n_0)=1$ this means that 
there does exist a $\Psi_i$-splitting 
above $\tau_{A,s}\ast \sigma_{n_0}$. 
Once again, choosing $r$ 
sufficiently large it follows that  $\rho_1 \ast \sigma_r$ 
is a string in $S$ extending $\rho_{A,s}$.
 
 We have shown that every initial segment of 
 $A$ has extensions in $S$. Since 
$S$ is a c.e.\ set, and $A$ is 
1-generic but does not have any initial segment 
in $S$, this gives the required 
contradiction. 
\end{proof}

\subsection{Measure}
The degrees which satisfy the join property
form a class of measure 1. Indeed, we show the following.

\begin{thm} \label{jointheo}
Every 2-random degree satisfies the join property.
\end{thm}
\begin{proof}
Suppose that $A$ is 2-random and
$\Psi^A=B$ for some incomputable set $B$
and a Turing functional $\Psi$.
We will exhibit a set $C<_T A$ such that 
$C\oplus B\equiv_T A$.
By Lemma \ref{coro:replfunat}, we may assume that
$\Psi$ is special.
In order to establish the existence of such 
a set $C$
it suffices to define
a computable procedure
which takes  a number $k\in \omega $ as input 
and returns (indices of) a $\emptyset'$-c.e.\ 
set of strings
$W$ with $\mu(W)<2^{-k}$ and a Turing 
functional $\Phi$ such that the following is satisfied for all sets $X$ which do not 
have a prefix in $W$: 
\begin{equation}\label{eq:basicreqfjoib}
\textrm{$\Psi^X$ is total}\Rightarrow  
(\textrm{$\Phi^X$ is total $\mathbf{\wedge}$ $X\leq_T \Psi^X \oplus \Phi^X$ 
$\mathbf{\wedge}$ 
$ X\not \leq_T \Phi^X$)}.
\end{equation}
 Since $A$ is 2-random there will be
 some $k\in\omega$ such that $A$ does not have a prefix in the
 set $W$ produced by the computable procedure with 
 input $k$. If we let $C=\Phi^A$ for the functional
 $\Phi$ that is produced by the procedure with input $k$, then $C$
 has the desired properties.

 Let us fix $k\in\omega$.
 The procedure on input $k$ will also produce the reduction $\Xi$ which 
establishes 
 $X\leq_T \Psi^X \oplus \Phi^X$ in (\ref{eq:basicreqfjoib}).  For ease of notation 
we let the oracle inputs for
$\Xi$ appear as arguments and not as superscripts.

Since 1-generic degrees do not bound 1-random degrees,
in order to ensure that $ A\not \leq_T \Phi^A$ it suffices to ensure that $\Phi^A$ is 
1-generic. 
We therefore look to satisfy the following 
requirements for all $X$ that do not have a prefix in
$W$, where $\{ W_e \}_{e\in \omega}$ is  an effective enumeration of all
upward closed c.e.\ sets of strings:
\[
R_e:\  \Psi^X \ \mbox{is total}\ \Rightarrow 
  \ \exists n [( \Phi^X\restr_n\in W_e)\ 
 \vee\ \forall \sigma\in W_e ( \Phi^X\restr_n \not\subseteq \sigma)]. 
\]

The construction fits the general description of
Section \ref{subse:asrdsdf}.
The purpose of an $e$-marker that is placed
on a string $\tau$ is
to enumerate axioms for  $\Phi$ and $\Xi$, and to ensure that   
 $R_e$ is satisfied for a fixed proportion of
the extensions $X$ of $\tau$. 
We describe only roughly how the marker operates 
now, the precise 
instructions will deviate just slightly 
from this rough description.  

 The marker begins by  searching for
a $\Psi$-splitting $(V,V')$ above $\tau$,  of $\tau$-measure 
$2^{-e}$.
Until such a splitting is found the marker is 
{\em inactive}. If and when the 
splitting is found, the marker 
becomes {\em active}. Upon finding the 
splitting the marker discards some 
strings from $V$ and $V'$, so that 
$V'$ is still of $\tau$-measure at least $2^{-(e+2)}$ and so that  
$\mu(V)/\mu(V\cup V')=2^{-e}$ (this may 
involve extending the length of the strings as necessary). Once active, the 
marker enumerates axioms for  $
\Phi$ and $\Xi$ on the strings in $V'$ 
and restrains the placement of markers on 
extensions of the strings in $V$.
Finally, if and when an
extension $\sigma$ of $\Phi^{\tau}$ appears in $W_e$,
it defines $\Phi^{\rho}$ to be an extension of $\sigma$ 
for all strings $\rho\in V$ and  lifts the 
restraint on the placement of markers on 
strings extending those in $V
$. In this event we say that the marker has  {\em acted}.

According to Lemma 
\ref{le:measplitbfa}, if the marker remains 
inactive then its actions may cause 
$\Phi^X$ to be partial although 
$\Psi^X$ is not partial,  for  $\tau$-measure at most 
$2^{-(e-1)}$.  Once the marker becomes 
active, it may cause $\Phi^X$ to be 
partial for those $X$ extending 
strings in $V$, but this is only $2^{-e}$ 
of the total proportion of strings in 
$V\cup V'$. Once active,  the 
marker ensures $R_e$ is satisfied for at least  a fixed proportion 
of the reals extending $\tau$, where this proportion  depends solely on $e$.

\ \paragraph{{\bf Construction of $\Phi$ and $\Xi$}}
At stage $0$ place a $k+4$-marker on the empty
string.

 At stage $s+1\in 2\omega^{[e]}+1$, if  
 $e>k+3$ then perform the following 
instructions, otherwise go to the 
next stage.  Order the strings on which 
$e$-markers sit, first by length and then 
from left to right. For each 
such $\tau$ and its marker in turn, 
perform the following instructions for the first 
of cases (a) and (b) which 
applies (or if neither case applies then do nothing).  
 
\begin{enumerate}
\item[(a)]  If the marker is {\em inactive} and
 there is a $\Psi$-splitting $(V,V')$ of $\tau$-measure
 $2^{-e}$ above $\tau$ in which the strings are of length $\leq s$ 
 then proceed as follows. Discard some 
 strings from $V$ and $V'$, so that $V'$ 
is still of $\tau$-measure at 
least $2^{-(e+2)}$ and so that  
$\mu(V)/\mu(V\cup V')=2^{-e}$ (we can assume 
the strings are long enough 
to do this).   
Take each $\rho\in V'$ in turn and enumerate the axioms  
$\Phi^{\rho}=\Phi^{\tau} \ast 0^{n_{\rho}}1$ 
 and $\Xi(\Psi^{\rho}, \Phi^{\rho})=\rho$, 
 where $n_{\rho}$ is chosen to be large at the time of the enumeration (and so 
increases as we proceed 
through the various $\rho$). Declare the marker to be {\em active}. 
 
\item[(b)] If the marker is {\em active} with splitting $(V,V')$ but has not acted
and there is some $\rho\in V'$ and some extension $\sigma$ of $\Phi^{\rho}$ in 
$W_e[s]$
 then proceed as follows. 
Choose the least such extension $\sigma$ and, taking each $\rho' \in V$ in turn, 
define 
$\Phi^{\rho'}=\sigma\ast 0^{n_{\rho'}}1$ and 
$\Xi(\Psi^{\rho'}, \Phi^{\rho'})=\rho'$,
 where $n_{\rho'}$ is chosen to be large at the time of the enumeration. Remove
any markers that sit on extensions of the strings in $V\cup V'$ and declare that 
the marker has {\em acted}. 
\end{enumerate}

At stage $s+1\in 2\omega+2$ let 
$\ell$ be large and proceed as follows for each 
string $\tau$ 
of length $\ell$ (starting from the 
leftmost string and moving right).
Let $\rho$ be the longest initial segment of $\tau$
on which a marker sits and let 
$e$ be the index of the marker.
If the $e$-marker is  active with 
splitting $(V,V')$ but has not acted 
and $\tau$ has a prefix in $V'$ then
place an $(e+1)$-marker on $\tau$.
If the $e$-marker is  active with 
splitting $(V,V')$ but has not acted 
and $\tau$ does not have a prefix in $V\cup V'$ then
place an $e$-marker on $\tau$.
If the $e$-marker has acted place an 
$e$-marker on $\tau$, unless $\tau$ has a 
prefix in $V$ in which case 
place an $(e+1)$-marker on $\tau$.
If a marker was placed on $\tau$, 
define $\Phi^{\tau}$ to be 
$\cup_{\rho\subset\tau} \Phi^{\rho}$ 
concatenated with 
$ 0^{n_{\tau}}1$, where 
$n_{\tau}$ is chosen to be large at 
the time of the enumeration.

\ \paragraph{{\bf Verification.}} It is clear 
that the axioms enumerated for $\Phi$ 
and $\Xi$ are consistent. The 
only point at which this condition could 
possibly be violated is when a marker on 
$\tau$ with splitting $(V,V')
$ acts and defines 
$\Phi^{\rho'}=\sigma\ast 0^{n_{\rho'}}1$ and 
$\Xi(\Psi^{\rho'}, \Phi^{\rho'})=\rho'$ for each 
$\rho'\in V$. Here $\sigma$ extends 
$\Phi^{\rho}$ for some $
\rho\in V'$ which is incompatible with each 
$\rho'\in V$. These axioms remain 
consistent with those 
previously enumerated, however,  precisely 
because $(V,V')$ is a $\Psi$-splitting. 

It is also clear that for each real $X$, one 
of the outcomes (1), (2) or (3) as 
described in Section 
\ref{subse:asrdsdf} must occur. Once an 
$e$-marker placed on $\tau$ becomes 
active, it ensures that at 
least a certain proportion of the reals 
extending $\tau$ do not have infinitely 
many $e$-markers placed on 
their initial segments, and so, as previously 
observed, it follows by the Lebesgue 
density theorem that the 
set of reals for which outcome (2) occurs 
is a $\Sigma^0_3$ set of measure 0. 
We may compute the index of 
a set of strings $S$ which is c.e.\ in 
$\emptyset'$, which is of measure 
$<2^{-k-1}$ and such that all reals for 
which outcome (2) occurs have a prefix in $S$. 

Now suppose that outcome (1) occurs for $X$. 
For any $e>k+3$ let $\tau$ be the 
longest initial segment of 
$X$ on which a permanent $e$-marker is placed. 
Let $(V,V')$ be the splitting for 
the marker placed on $\tau$. 
Suppose the marker on $\tau$ does not act and 
$X$ extends a string in $V'$. 
In this case $R_e$ is 
satisfied and the lengths of $\Phi^X$ and 
$\Xi(\Psi^X,\Phi^X)$ are properly 
increased by the marker on $\tau$. 
Otherwise the marker acts and 
$X$ extends a string in $V$, 
but this allows us 
to draw the same 
conclusion. 

It remains to show that we can find the index of 
a set of strings $V$ which is c.e.\ 
in $\emptyset'$, such that $\mu(V)\leq 2^{-k-1}$, 
and such that any $X$ 
for which outcome (3) occurs either 
has $\Psi^X$ partial, or else 
has an initial segment in $V$. 
We can then put $W=V\cup S$. 
So consider  the set of strings $\tau$ that 
hold a permanent  marker which remains inactive. 
This is a prefix-free set. For 
each $\tau$ in the set, if $e$ 
is the index of the marker that sits on $\tau$ 
then we can (uniformly) find the 
index of a set of strings of 
$\tau$-measure $\leq 2^{-e+1}$ which contains an 
initial segment of any  extension of 
$\tau$ on which $\Psi$ is 
total. Since we only consider $e>k+3$, taking 
the union over all such $\tau$ 
gives a set of measure $\leq 2^{-(k+2)}$. 

Next, fix $e>k+3$ and consider all 
those $\tau$ on which a permanent 
$e$-marker is placed, which is 
eventually active but does not act. If $(V_0,V_0')$ 
is the splitting corresponding 
to one such $\tau$ and 
$(V_1,V_1')$ is the splitting corresponding to a 
different one, then any string in 
$V_0\cup V_0'$ is 
incompatible with any string in $V_1\cup V_1'$. 
Since the measure of $V$ is 
always $2^{-e}$ of the total 
measure of $V\cup V'$, the measure of the 
union of all  corresponding sets $V$ 
is at most $2^{-e}$. Taking 
the union over all $e>k+3$, we obtain a 
set of measure $< 2^{-(k+3)}$ as 
required.  
\end{proof}

To what extent is Theorem \ref{jointheo} optimal? 
It is not too difficult to show that 
there exist Demuth 
randoms that do not satisfy the join property. 
This follows from the result of 
\cite{AL11} that all low fixed point 
free degrees fail to satisfy the join property, 
and the fact \cite[Theorem 3.6.25]
{Ottobook} that there exist low 
Demuth random reals. The following question remains open: 

\begin{question}
Does there exist a weakly 2-random 
degree which does not satisfy the join 
property? 
\end{question}

Next we use a very slightly modified 
version of the machinery developed in 
Section \ref{subse:allnzdbba} in 
order to prove another instance of 
our heuristic principle. The original machinery 
could certainly have been 
specified in such a way that no 
modification would be required for this 
application, but this would have made 
the proof of Theorem \ref{th:2ranb1gen} 
seem more complicated.  

\begin{thm}\label{th:2randowndenjoinp}
Every degree that is bounded 
by a 2-random degree satisfies
the join property.
\end{thm}
\begin{proof}
Suppose that $A$ is a 2-random 
set that computes an incomputable set $B$ via  $\Theta$. 
We need to show that $B$ has the join property.
If $B$ is of  1-generic degree then the theorem holds
by Theorem \ref{th:1genjoin},
so suppose otherwise.
By Lemma \ref{coro:replfunat} we may assume that $\Theta$
is {\em special}.
Suppose that $B$ computes an incomputable set $C$
via a Turing functional $\Psi$.
By Lemma \ref{coro:replfungenat} we may assume that
$\Psi$ is {\em special}.
In order to show that there is some $D<_T B$ such that 
$D\oplus C\equiv_T B$, it suffices
to define
a computable procedure
which takes 
a number $k$ and returns (indices of) a $\emptyset'$-c.e.\ 
set of strings
$W$ with $\mu(W)<2^{-k}$, and a Turing 
functional $\Phi$ such that the following holds for all sets $X$ which do not have 
a prefix in $W$:
\begin{equation}\label{eq:abasicreqfjoib}
\textrm{$\Theta^X=Y$ and $\Psi^Y$ is total}\Rightarrow  
\textrm{$\Phi^Y$ is total $\mathbf{\wedge}$ $Y\leq_T \Psi^Y \oplus \Phi^Y$ 
$\mathbf{\wedge}$ 
$ Y\not \leq_T \Phi^Y$}.
\end{equation}
 \noindent In order to see that this suffices, consider the sequence of procedures
 with input $k\in\omega$. Since $A$ is 2-random,
 for some $k\in\omega$ the corresponding procedure will produce $\Phi$
 such that the right hand side of the implication in 
 (\ref{eq:abasicreqfjoib}) holds with
 $Y=B$. In other words, 
 $D\oplus C\equiv_T B$ and $D<_T B$ where $D=\Phi^B$.
 Actually,  since we assumed that $\Theta^A$ is not of 1-generic degree, 
it suffices to replace $Y\not \leq_T \Phi^Y$ in
(\ref{eq:abasicreqfjoib}) with the requirement that 
$\Phi^Y$ is 1-generic.
 
 It remains to define and verify this procedure with input
 $\Theta,\Psi$ and $k\in\omega$.  The procedure will also produce
 a Turing functional
 $\Xi$ for the reduction 
 $Y\leq_T \Psi^Y \oplus \Phi^Y$ in (\ref{eq:abasicreqfjoib}).
We look to satisfy the following requirements for all $X$ which do not have a 
prefix in $W$:
\[
R_e:\ \textrm{$\Theta^X=Y$ and $\Psi^Y$ is total}\Rightarrow\ \   
\left\{\parbox{6cm}{$\Phi^Y$ is total\ \   and\ \   
$\Xi(\Psi^Y, \Phi^Y)=Y$ \ \  and \ \  \\ 
  $\exists n\ \big[\Phi^Y\restr_n\in W_e\ 
 \vee\ \forall \eta\in W_e,\ \Phi^Y\restr_n \not\subseteq \eta\big]$}\right.
\] 
\noindent  where $\{ W_e\}$ is  an effective enumeration of all
upward closed c.e.\ sets of strings. 
Note that for ease of notation we let the oracles in
$\Xi$ appear as arguments and not as superscripts. We define a construction 
which deviates only slightly 
from the framework described in Section \ref{subse:allnzdbba}. Just as described 
there, markers are initially 
inactive, but now submarkers are also initially inactive and must wait to be made 
active. In defining the 
construction we make use of the following inequalities:

\begin{eqnarray}
&\pi(T_{\sigma})[s]\geq 2^{-k-2}\cdot\pi^{\ast}(\sigma)[s].\label{activemcon}\\
&\pi(\rho)[s]<2^{-q_{\sigma'}}.\label{smalldiv}\\
&0\leq \pi(F_{\sigma}(\sigma'))[s]-2^{-e}\cdot\pi(P_{\sigma}(\sigma'))[s]<2^{-
q_{\sigma'}}. \label{niceslice}
\end{eqnarray}

\ \paragraph{{\bf Construction of $\Phi,\Xi$}}
At Stage 0 place a k+4-marker on the
empty string.

At stage $s+1\in2\omega^{[e]}$, if  $e>k+3$  then
consider each string $\sigma$ on which an 
$e$-marker sits in turn (ordered first 
by length and then from left 
to right), and proceed according to the first case below that applies.

\begin{enumerate}
\item If (\ref{eq:aaproxofpi}) does not hold, let
$\pi^{\ast}(\sigma)=\pi(\sigma)[s]$, declare that
the $e$-marker on $\sigma$ is {\em injured} and
is inactive.  Remove any markers and submarkers that sit on
proper extensions of $\sigma$.  Let $m_{\sigma}$
be large and place a submarker on each extension of $\sigma$ of length 
$m_{\sigma}$.

\item Otherwise, if the marker is inactive and (\ref{activemcon}) holds, where 
$T_{\sigma}$ is the set of all 
strings extending $\sigma$ of length $m_{\sigma}$, then declare the marker to 
be active and define 
$s_{\sigma}=s$.

\item If the marker is already active, then 
proceed as follows for each submarker 
placed on a string $
\sigma'$ by $\sigma$, according to the 
first case below which applies.
\begin{enumerate}
\item[(a)] If the submarker is inactive and 
there exists a $\Psi$-splitting $(U,V)$ 
above $\sigma'$ such that $
\pi(U)[s]\geq \pi(V)[s]\geq 2^{-e}\pi(\sigma')[s_{\sigma}]$ and such that 
(\ref{smalldiv}) holds  for all $\rho\in U
\cup V$, then declare the submarker to be active. 
In this case let  $F_{\sigma}(\sigma')$ be a subset of $U$ such that 
(\ref{niceslice}) holds,  defining  
$P_{\sigma}(\sigma')=F_{\sigma}(\sigma')\cup V$. 
Take each $\rho\in V$ in turn 
and enumerate the axioms  
$\Phi^{\rho}=\Phi^{\sigma'} \ast 0^{n_{\rho}}1$ 
 and $\Xi(\Psi^{\rho}, \Phi^{\rho})=\rho$, 
 where $n_{\rho}$ is chosen to be large at 
 the time of the enumeration (and so 
increases as we proceed 
through the various $\rho$).

\item[(b)]  If the submarker is {\em active}  but has not acted
and there is some $\rho\in P_{\sigma}(\sigma')-F_{\sigma}(\sigma')$ and some 
extension $\eta$ of $
\Phi^{\rho}$ in $W_e[s]$
 then proceed as follows. 
Choose the least such extension $\eta$ and, taking each 
$\rho' \in F_{\sigma}(\sigma')$ in turn, define 
$\Phi^{\rho'}=\eta\ast 0^{n_{\rho'}}1$ and 
$\Xi(\Psi^{\rho'}, \Phi^{\rho'})=\rho'$,
 where $n_{\rho'}$ is chosen to be large 
 at the time of the enumeration. Remove
any markers that sit on extensions of 
the strings in $P_{\sigma}(\sigma')$ and 
declare that the submarker 
has {\em acted}.

\item[(c)]  If the previous cases do not apply 
and the second inequality of 
(\ref{niceslice}) no longer holds 
then there are two possibilities to consider.  
If (\ref{smalldiv}) still holds for all 
$\rho\in F_{\sigma}(\sigma')$,  
then remove strings from $F_{\sigma}(\sigma')$ 
so that (\ref{niceslice}) holds. If 
not then choose $\ell$ to be 
large, and replace each string   
$\rho\in F_{\sigma}(\sigma')$ with all extensions 
of $\rho$ of length $\ell$, to 
form a new $F_{\sigma}(\sigma')$ (whenever we 
redefine $F_{\sigma}(\sigma')$ 
we also consider 
$P_{\sigma}(\sigma')$ to be 
redefined accordingly, 
$P_{\sigma}(\sigma')=F_{\sigma}(\sigma')\cup V$). 

\end{enumerate} 

\end{enumerate}

At stage $s+1\in 2\omega+1$ 
let $\ell$ be large  and
do the following for each string  $\rho$ of length $\ell$. 
Let
  $\sigma$ be the longest initial segment of $\rho$
on which a marker sits. 
Let $\sigma'$ be the string of length 
$m_{\sigma}$ which is an initial segment of  
$\rho$,
and let $e$ be the index of the marker 
placed on $\sigma$. If the submarker 
placed on $\sigma'$ is not 
active, then we do not place any 
marker on $\rho$, so suppose otherwise. 
If the submarker on $\sigma'$ has not acted
and $\rho$ has a prefix in 
$P_{\sigma}(\sigma')-F_{\sigma}(\sigma')$ then
place an $(e+1)$-marker on $\rho$.
If the submarker on $\sigma'$  has not acted 
and $\rho$ does not have a prefix in $P_{\sigma}(\sigma')$ then
place an $e$-marker on $\rho$.
If the submarker has acted place an 
$e$-marker on $\rho$, unless $\rho$ has a 
prefix in $F_{\sigma}
(\sigma')$, in which case place an $(e+1)$-marker on $\rho$.
If a marker was placed on $\rho$, 
define $\Phi^{\rho}$ to be 
$\cup_{\rho'\subset\rho} \Phi^{\rho'}$ concatenated with 
$ 0^{n_{\rho}}1$, where 
$n_{\rho}$ is chosen to be large at the time of the enumeration.

\ \paragraph{{\bf Verification}} The question 
of consistency for $\Phi$ and $\Xi$ is 
only trivially different than 
the case for Theorem \ref{jointheo}. 
We are therefore left to specify $W$ such 
that $\mu(W)<2^{-k}$ and $W$ 
has an initial segment of every  $X$ such that 
$\Theta^X=Y$, $\Psi^Y$ is total 
and either outcome (2) or (3) 
holds for $Y$.
  First of all consider those $\Theta^X=Y$ for 
  which outcome (3) applies. There 
are three possibilities. First, it 
may be the case that a permanent marker is 
placed on $\sigma \subset Y$, 
which never becomes active.  By 
Lemma \ref{le:cconran} we can find the index 
for a $\emptyset'$-c.e.\ set of 
strings $V_0$ such that $
\mu(V_0)<2^{-k-2}$ and $V_0$ contains an 
initial segment of every $X$ for 
which $\Theta^X$ is total and 
has such a marker placed on an initial segment. 
The second possibility is that the 
first case does not apply 
but a permanent submarker is placed on an 
initial segment of $\Theta^X$ which 
never becomes active. 
Since the strings on which such submarkers 
are placed form a prefix-free set and 
we only work with $e>k
+3$,  Lemma \ref{le:mediwasplitbfa}  directly 
provides us with a set $V_1$ such 
that $\mu(V_1)\leq 2^{-k-3}$ 
and which contains an initial segment of every 
$X$ such that $\Theta^X=Y$ is 
total, $\Psi^Y$ is total, and 
such that  such a submarker is placed on an 
initial segment of $Y$. The last 
possibility is that $\Theta^X$ 
extends a string in (the final value) 
$F_{\sigma}(\sigma')$ for some permanent 
submarker which does not act 
and which is placed by an $e$-marker on $\sigma$. 
Since, for fixed $e$,  the 
union of all the various 
$P_{\sigma}(\sigma')$ corresponding to such 
submarkers forms a prefix-free set, 
and since we maintain the 
second inequality of (\ref{niceslice}) it follows that, 
summing over all $e>k+3$,  
we can find the index for an $
\emptyset'$-c.e.\ set of strings $V_2$ such that 
$\mu(V_2)<2^{-k-2}$ and $V_2$ 
contains an initial segment 
of every $X$ for which $\Theta^X$ extends a 
string in one of these $F_{\sigma}(\sigma')$. 
  
  Finally, we must show that the set of $X$ such 
  that $\Theta^X$ is total and has 
outcome (2) is a $\Sigma^0_3$ set of measure 0.  
Now suppose that a permanent marker is placed 
on $\sigma$ which 
becomes active at stage $s_{\sigma}$. 
 We wish to find a prefix-free set of strings 
$V_{\sigma}$ extending $
\sigma$ such that $\pi(V_{\sigma})$ is at least a 
fixed proportion of $\pi(\sigma)$ 
and no $e$-markers are 
placed on strings extending those in $V_{\sigma}$. 
Then the result will follow by 
Lemma \ref{le:extled}.  
Subsequent to the last injury of the marker on $\sigma$ we maintain 
(\ref{eq:aaproxofpi}), and activation of 
the marker requires that (\ref{activemcon}) holds. If the marker places a 
permanent submarker on $\sigma'$ 
which does not become active, then no markers 
will be placed on extensions of $\sigma'$, so we can 
immediately enumerate all such 
$\sigma'$ into $V_{\sigma}$. Now we consider 
each of the $\sigma'$ on 
which the marker places a permanent submarker 
which becomes active, and we 
look to enumerate a set of 
strings $D_{\sigma}(\sigma')$ into 
$V_{\sigma}$, such that all these strings 
extend $\sigma'$ and $\pi(D_{\sigma}(\sigma'))$ is at least a fixed 
proportion of $\pi(\sigma')[s_{\sigma}]$. We consider 
approximations to $D_{\sigma}(\sigma')$ and 
then take the final value. At each 
stage define $D_{\sigma}
(\sigma')$ by replacing each string in 
$F_{\sigma}(\sigma')$ with the shortest 
initial segment which is 
incompatible with all strings that are not in
$F_{\sigma}(\sigma')$ (so this set 
changes as $F_{\sigma}(\sigma')$ does).  
  Now at stage $s_0$ at which the submarker is 
  activated, we have that 
  $\pi(P_{\sigma}(\sigma')-F_{\sigma}(\sigma'))[s_0] 
  \geq 2^{-e-1}\cdot \pi(\sigma')[s_{\sigma}]$, and by (\ref{niceslice}) 
we therefore have that $\pi(D_{\sigma}(\sigma'))[s_0]\geq 
2^{-2e-1}\cdot \pi(\sigma')[s_{\sigma}]$. We 
wish to show by induction  that 
this condition is maintained at subsequent stages. First note that the strings in 
$P_{\sigma}(\sigma')-F_{\sigma}(\sigma')$ 
do not subsequently change. When we redefine 
$F_{\sigma}(\sigma')$ by extending 
the length of the strings, this does not change $D_{\sigma}$. When we remove 
strings from $F_{\sigma}(\sigma')$ 
at a stage $s$ we maintain satisfaction of the first inequality in 
(\ref{niceslice}) so that, since 
$\pi(P_{\sigma}(\sigma')-F_{\sigma}(\sigma'))[s]\geq  
\pi(P_{\sigma}(\sigma')- F_{\sigma}(\sigma'))[s_0]$,    
$\pi(D_{\sigma}(\sigma'))[s_0]\geq 2^{-2e-1}\cdot \pi(\sigma')[s_{\sigma}]$ 
still holds. 
 \end{proof}

\begin{coro}
Every non-zero degree
below a 2-random degree is the
supremum of two lesser degrees. Hence 
2-random degrees do not bound strong minimal covers.
\end{coro}
\begin{proof}
This is a consequence of Theorem 
\ref{th:kur2ranb1gen} and Theorem \ref{th:2randowndenjoinp}.
\end{proof}

\section{Being the top of a diamond}
We say that a Turing degree $\bf{c}$ is the top 
of a diamond if there exist $\bf{a},\bf{b}<\bf{c}$ 
such that $\boldsymbol{a}\vee \boldsymbol{b}=\boldsymbol{c}$ and 
$\boldsymbol{a}\wedge \boldsymbol{b}=\boldsymbol{0}$. 
As will be discussed in the 
following sections, all sufficiently generic degrees satisfy 
the complementation property, 
which is a strictly stronger condition than being 
the top of a 
diamond so long as the degree 
concerned is not $\bf{0}$ or minimal. 
Since we do not know the measure of the 
degrees which satisfy the complementation 
property or even the meet property, however, it is interesting to 
consider the property of being the top of a diamond for the measure-theoretic 
case. 

It is well known that every 2-random degree is the top of a diamond.
This is a simple consequence of van Lambalgen's theorem that we mentioned in
Section \ref{subse:motiv} and the result in \cite{MR2352724} that was discussed 
in the proof of Lemma 
\ref{coro:replfunat}. 
We show that, in fact, the same property is
shared by all nontrivial degrees with a 2-random upper bound.

\begin{thm}Every non-zero degree that is 
bounded by a 2-random degree is the 
join of a minimal pair of 1-generic degrees. 
\end{thm}
\begin{proof}
Assume that $C= \Theta^D$ where $D$ is 2-random and $C$ is incomputable. 
It follows from Theorem \ref{th:2ranb1gen} and (the proof of) 
Theorem \ref{th:2randowndenjoinp} that $C$ is 
the join of two 1-generic sets. Here we will show 
that $C$ is the join of two 1-generic sets which form a 
minimal pair. We will construct the 
1-generic sets via two functionals $\Phi$
 and $\Psi$. As before we may 
assume that $\Theta$ is special. Given $k\in \omega$ we define a construction 
which suffices to specify the 
index for a $\emptyset'$-c.e.\ set of strings $W$, such that $\mu(W)<2^{-k}$, 
and such that for all $X$ which 
do not have a prefix in $W$ and such that $\Theta^X=Y$ is total  the following 
requirements are satisfied:
\[\mbox{For all }e \in 3\omega +1, \;\;
R_e: \ \exists n\ \big[\Phi^Y\restr_n\in W_\frac{e-1}{3}\ \vee\ \forall \sigma\in 
W_\frac{e-1}{3},\ \Phi^Y\restr_n 
\not\subseteq \sigma\big];
\]
\[\mbox{For all }e \in 3\omega +2, \;\;
R_e: \ \exists n\ \big[\Psi^Y\restr_n\in W_\frac{e-2}{3}\ \vee\ \forall \sigma\in 
W_\frac{e-2}{3},\ \Psi^Y\restr_n 
\not\subseteq \sigma\big].
\]

We also need to make $\Phi^C$ and $\Psi^C$ a minimal pair. A standard 
approach to building a minimal 
pair of sets is to use an approximation via finite strings $\{\alpha_s\}$ and 
$\{\beta_s\}$ with $A = \lim_s \alpha_s$ and $B=\lim_s\beta_s$.  
In order to ensure that $\Psi_d^A$ and $\Psi_e^B$ 
do not both compute 
the same incomputable set, at some stage $s$, we look for 
$\alpha' \supseteq \alpha_s$, $\beta' \supseteq \beta_s$ 
and $m\in \omega$ such that
\begin{equation}
\label{eq:incompat_ext}
\Psi_d^{\alpha'}(m)\de \ne \Psi_e^{\beta'}(m)\de.
\end{equation}
If such a pair of extensions are found we set 
$\alpha_{s+1}=\alpha'$ and $\beta_{s+1}=\beta'$. 
Failure to find such extensions 
implies that if $\Psi_d^A=\Psi_e^B$ is total, 
then it is computable. By using Posner's trick it suffices to 
meet the following requirements 
for all $e \in 3\omega$ and for all $X$ which 
do not have a prefix in $W$ and 
such that $\Theta^X=Y$ is 
total:

\[
R_{ e }:\ \ \Psi_{\frac{e}{3}}(\Phi^Y)\ \mbox{is not total, or}\    
\Psi_{\frac{e}{3}}(\Phi^Y) \ne \Psi_{\frac{e}{3}}
(\Psi^Y), \   \mbox{or} \    \Psi_{\frac{e}{3}}(\Phi^Y) \   \mbox{is computable}. 
\]

Note that here, for ease of notation, 
we sometimes let oracle inputs appear as 
arguments rather than 
suffixes. We say that a string $\rho$ is an 
{\em $e$-failure at stage $s$}, if there 
exist  $\rho_1, \rho_2$ 
extending $\rho$, such that for some $n$:
\[\Psi_{\frac{e}{3}}(\Phi^{\rho_1})[s]\restr_n =
\Psi_{\frac{e}{3}}(\Psi^{\rho_1})[s]
\restr_n \ne
\Psi_{\frac{e}{3}}(\Phi^{\rho_2})[s]\restr_n=
\Psi_{\frac{e}{3}}(\Psi^{\rho_2})[s]
\restr_n.\]
Note that if $e \in 3 \omega$ and $\rho$ is 
not an $e$-failure at any stage, then 
requirement $R_e$ is 
achieved on all extensions of $\rho$. 

If, for $e \in 3 \omega$, we place an 
$e$-marker on a string $\sigma$, and a 
submarker on $\sigma'$ then 
$F_\sigma(\sigma')$ is the set of strings 
extending $\sigma'$  which we may 
think of as guessing that a pair 
of  extensions can be found as per 
\eqref{eq:incompat_ext}. However, we also 
need to ensure that $Y$ is 
computable in the join of $\Psi^Y$ and $\Phi^Y$. 
Assume that at some stage 
$s$, we have $\rho \in F_\sigma(\sigma')$ 
and  $\Psi^\rho=\alpha$ and $\Phi^\rho=\beta$.  We look for 
extensions of $\alpha$ and 
$\beta$ on which we can achieve our 
requirement but also on which we can 
encode $\rho$.  For any two 
strings  $\rho_0, \rho_1 \in P_\sigma(\sigma') - F_\sigma(\sigma')$, 
we will ensure that $\rho_0$ and $
\rho_1$ are both $\Phi$-splitting and $\Psi$-splitting 
(this is easily achieved since we control these 
functionals).   If $\rho_0$ and  $\rho_1$ are both $e$-failures, then let 
$\alpha'=\Phi^{\rho_0}$ and 
$\beta'=\Psi^{\rho_1}$. We can ensure that 
when the submarker on $\sigma'$ goes to act 
we have $\alpha \subseteq 
\alpha'$ and $\beta \subseteq \beta'$. 
Up until this point, there has been no need 
to encode anything into 
the join of  $\alpha'$ and $\beta'$. 
Hence, at this point,  we could define  
$\Phi^\rho \supseteq \alpha'$ and 
$\Psi^\rho \supseteq \beta'$ and then  set some extension of the  
join of $\alpha'$ and $ \beta'$ to 
compute $\rho$.
Now by the $e$-failure condition there is an $n$, and  $\alpha_0$ and $
\alpha_1$ extending $\alpha'$ such that  
$\Psi_{\frac{e}{3}}(\alpha_0)\restr_n \ne \Psi_{\frac{e}{3}}(\alpha_1)\restr_n$. 
Additionally there is an $m$, 
and $\beta_0$ and $\beta_1$ extending $\beta'$ such that  
$\Psi_{\frac{e}{3}}(\beta_0)\restr_m \ne \Psi_{\frac{e}{3}}(\beta_1)\restr_m$. 
Hence we can find $i,j \in \{0,1\}$ 
such that 
$\Psi_{\frac{e}{3}}(\alpha_i)\restr_{\min(n,m)} \ne 
\Psi_{\frac{e}{3}}(\beta_j)\restr_{\min(n,m)}$.
Thus we can achieve success on all strings 
$\rho \in F_\sigma(\sigma')$ by 
defining $\Phi$ and $\Psi$ on 
these strings to extend $\alpha_i \ast \rho$ and 
$\beta_j\ast \rho$ respectively.

We make use of the following inequalities:

\begin{eqnarray}
& \pi(\sigma)[s]<2\pi^{\ast}(\sigma)[s].\label{reset}\\
Ê&\pi(P_{\sigma})[s] \geq 2^{-k-2}\cdot \pi^{\ast}(\sigma)[s]\ \  \  \mbox{and} \ \ \ \ 
\forall \rho\in P_{\sigma}(\sigma')[\pi(\rho)[s]<2^{-q_{\sigma'}}].\label{active}\\
&0\leq \pi(F_{\sigma}(\sigma'))[s] 
-2^{-e}\cdot \pi(P_{\sigma}(\sigma'))[s_{\sigma}]\ <\   2^{-q_{\sigma'}}.
\label{nicslic}
\end{eqnarray}

\ \paragraph{{\bf Construction of $\Phi$ and $\Psi$}}
At Stage $0$ place a $k+4$-marker on the empty string. 
At stage $s+1 \in 2\omega^{[e]}$, if $e >k+3$, then for each $e$-marker that sits 
on a string
$\sigma$, proceed according to the first case below that applies:
\begin{enumerate}
\item If (\ref{reset}) does not hold then redefine
$\pi^{\ast}(\sigma)=\pi(\sigma)[s]$. If the  $e$-marker on $\sigma$ is currently 
{\em active}, then declare the 
marker to be {\em inactive}.  For all strings $\rho \in F_\sigma(\sigma')$, 
define
$\Phi^{\rho}$ to be $\cup_{\rho'\subset\rho} \Phi^{\rho'}$ 
concatenated with $\rho$, and define 
$\Psi^{\rho}$ to be $\cup_{\rho'\subset\rho} \Psi^{\rho'}$ 
concatenated with $\rho$.
Remove any markers and submarkers that sit on
proper extensions of $\sigma$.  Let $m_{\sigma}$
be large and place a submarker on each extension of $\sigma$ of length 
$m_{\sigma}$.

\item  Otherwise, if the marker is inactive and 
(\ref{active}) holds for some set of strings
$P_{\sigma}(\sigma')$ for each submarker $\sigma'$, where the strings in 
$P_{\sigma}(\sigma')$ are all 
those extending $\sigma'$ of a certain length, 
declare that the marker is {\em active} and define $s_{\sigma}=s$.
For each submarker $\sigma'$, define 
$F_{\sigma}(\sigma')$ to be the least initial segment of $P_\sigma(\sigma')$ 
under the lexicographical 
ordering such that (\ref{nicslic}) holds.

\item If the marker is active then for each 
submarker $\sigma'$ of $\sigma$ which 
has not acted perform the 
following tasks:
\begin{enumerate}
\item If (\ref{nicslic}) does not hold there are two possibilities. If the second 
inequality of (\ref{active})  holds 
when we only allow the quantifier to range 
over strings in $F_{\sigma}(\sigma')$,  
then remove strings from 
$F_{\sigma}(\sigma')$ so that (\ref{nicslic}) does hold.  
Otherwise let $\ell$ be 
large and replace each string 
in $F_{\sigma}(\sigma')$ with all extensions of length $\ell$. 

\item For
each extension
$\rho$ of $\sigma'$ in $P_{\sigma}(\sigma')-F_{\sigma}(\sigma')$,
define
$\Phi^{\rho}$ to be $\cup_{\rho'\subset\rho} \Phi^{\rho'}$ 
concatenated with $\rho$.
\item For each extension
$\rho$ of $\sigma'$ in $P_{\sigma}(\sigma')-F_{\sigma}(\sigma')$,
define $\Psi^{\rho}$ to be 
$\cup_{\rho'\subset\rho} \Psi^{\rho'}$ concatenated 
with $\rho$.

\item If $e\in 3\omega$ and  $\rho\in P_\sigma(\sigma')$ 
is an $e$-failure at the 
current stage, but has not 
been so at any previous stage in $2\omega^{[e]}$ 
since the marker on $\sigma$ 
was last made active, then 
remove all markers from $\rho$ and any extensions.

\item  If $e\in 3\omega$ and there exist two distinct strings 
$\rho_1, \rho_2 \in P_\sigma(\sigma') - F_\sigma(\sigma')$ 
such that $\rho_1$ and $\rho_2$ are both $e$-failures then 
there must exist strings 
$\rho'_1 \supseteq \rho_1$ and $\rho_2' \supseteq \rho_2$ such that 
$\Psi_{\frac{e}{3}}(\Phi^{\rho_1'})[s]$ and 
$\Psi_{\frac{e}{3}}(\Psi^{\rho_2'})[s]$ are 
incomparable.
For all $\rho \in F_\sigma(\sigma')$  define: 
\[\Phi^\rho[s+1] = \Phi^{\rho_1'}[s] \ast\rho \mbox{ and }\Psi^\rho[s+1] = 
\Psi^{\rho_2'}\ast \rho.\]
Declare that the submarker on $\sigma'$ has 
acted and  remove all markers and 
submarkers that sit on 
proper extensions of $\sigma'$.

\item If $e \in 3\omega +1$, and $\sigma'$ can 
act because there exists a string 
in  $W_\frac{e-1}{3}$ 
extending $\Phi^{\sigma'}$, then for all 
$\rho \in F_\sigma(\sigma')$ define
$\Phi^\rho$ as per Theorem \ref{th:2ranb1gen} 
but define $\Psi^\rho$ to be $
\cup_{\rho'\subset\rho} \Psi^{\rho'}$ 
concatenated with $ \rho$. Similarly for the 
case $e\in 3\omega +2$.

\end{enumerate}
\end{enumerate}
At stage $s+1 \in 2 \omega +1$ let $\ell$ be large. 
For each string $\rho$ of 
length $\ell$ find the longest 
initial segment $\sigma$ with a marker. Let $e$ be 
such that marker on $\sigma$ 
is an $e$-marker. If the 
marker is inactive, then do not place a marker on 
$\rho$. If the marker is active 
let $\sigma' \subseteq \rho$ 
be the unique string on which there sits a submarker of 
the $e$-marker on $\sigma
$. If the submarker has acted then if $\rho$ extends a string in 
$F_\sigma(\sigma')$ place an $(e+1)$-marker 
on $\rho$, otherwise place an 
$e$-marker.
If the marker has not acted then let $\sigma''$ be the 
unique initial segment of $\rho$ in $P_\sigma(\sigma')$. 
If $\sigma'' \in F_\sigma(\sigma')$ 
then do not place a marker on $\rho$. If 
$\sigma''\in P_\sigma(\sigma')-F_\sigma(\sigma')$ 
is an $e$-failure then place an $e$-marker on $\rho$, 
otherwise place an $(e+1)$-marker.

\ \paragraph{{\bf Verification}} 

The analysis of outcomes (2) and (3) occurs 
exactly as in the proof of Theorem 
\ref{th:2ranb1gen} with one 
small adjustment. Suppose $e\in 3\omega$ and let 
 $T$ be the set of strings on which we place 
 permanent submarkers which do not 
act, which are placed by 
permanent $e$-markers which are eventually always active. 
Let $J$ be the union of all (the final values) $F_{\sigma}(\sigma')$
such that $\sigma' \in T$ and the submarker 
on $\sigma'$ is placed by a marker 
on $\sigma$. 
For any $\sigma' \in T$, let 
$S(\sigma')=\{ \rho : \rho \in P_{\sigma}(\sigma') 
-F_{\sigma}(\sigma')$ and $\rho$ 
is not an $e$-failure$\}$. No extension of any string in $S(\sigma')$ has an 
$e$-marker placed on it. Since  
the submarker never acts, there is at most one 
string in $P_{\sigma}(\sigma')$ 
which is an $e$-failure, and 
so $S(\sigma')$  contains all the initial elements of 
$P_{\sigma}(\sigma') -F_{\sigma}(\sigma')$ with the 
possible exception of one  string that is an $e$-failure. 
The fact that we maintain 
(\ref{nicslic}) therefore 
means that 
$\pi(F_{\sigma}(\sigma'))-2^{-e}\cdot 
\pi(S_{\sigma}(\sigma')\cup F_{\sigma}(\sigma'))<2\cdot 2^{-q_{\sigma'}}$. 
The union of all $F_{\sigma}(\sigma') \cup S_{\sigma}(\sigma')$ as $\sigma'$ 
ranges over the elements of 
$T$ forms a prefix-free set. 
This suffices to show that $\pi(J)$ is sufficiently small. 

We now consider those $Y$ for which outcome (1) 
occurs. In order to show that 
all genericity and minimal 
pair requirements are satisfied with respect to 
$Y$, for each $e>k+3$ consider 
the longest initial segment of 
$Y$ on which a permanent submarker is 
placed by a permanent $e$-marker. 
Either the submarker does not 
act and $Y$ extends a string in 
$P_{\sigma}(\sigma')-F_{\sigma}(\sigma')$, which 
is not an $e$-failure if 
$e\in 3\omega$, or else the submarker acts and 
$Y$ extends a string in 
$F_{\sigma}(\sigma')$. In either case 
the requirement is satisfied. Finally we 
need to show that $Y$ is computable in 
the join of 
$\Phi^Y$ and $\Psi^Y$. Recall that if a 
marker is placed on $\sigma$, then at any 
stage $P_{\sigma}$ is the 
union of all the various $P_{\sigma}(\sigma')$ 
for submarkers $\sigma'$. First 
note that if 
 $\rho$ is in some $ P_\sigma$, then the 
 construction will enumerate at most one
 $\Phi^\rho$ axiom and at most one $\Psi^\rho$ 
 axiom. Secondly, this is the only 
way in which the 
construction enumerates $\Phi$ and $\Psi$ axioms.

\begin{lem}
At any stage, if $\rho_0$ and $\rho_1$ are 
distinct elements of $P_\sigma$ on 
which both $\Phi$ and $\Psi$ 
are defined, then either $\Phi^{\rho_0}$ is incompatible with 
$\Phi^{\rho_1}$ or $\Psi^{\rho_0}$ is 
incompatible with $\Psi^{\rho_1}$.
\end{lem}
\begin{proof}
First assume that this is the first $P_\sigma$ defined for the 
$e$-marker on $\sigma$. 
Let $\alpha = \bigcup_{\sigma'\subset\sigma}\Phi^{\sigma'}$ and 
$\beta = \bigcup_{\sigma'\subset\sigma}\Psi^{\sigma'}$.  
In this case we have that no axioms have been 
enumerated for any $\rho'$ with 
$\sigma \subseteq \rho' \subset \rho_0$ or  
$\sigma \subseteq \rho' \subset \rho_1$. 
If $e \in 3 \omega+1$, then  this implies that 
$\Psi^{\rho_0} = \beta \ast \rho_0$ 
and $\Psi^{\rho_1} = \beta \ast 
\rho_1$. The case for $e \in 3 \omega+2$ is 
similar with $\Phi$ in place of $\Psi$.
If $e \in 3 \omega$, then we  have 
$\Phi^{\rho_i} = \alpha \ast \rho_i$ unless 
$\rho_i \in F_\sigma(\sigma')$ at 
some stage when the submarker on 
$\sigma'$ acts. Hence we only need to 
consider the case when at least 
one string has this property. Assume $\rho_0$ 
has this property. We have that:
$\Phi^{\rho_0} = \alpha \ast \rho_2\ast \rho \ast \rho_0$ and
$\Psi^{\rho_0} = \beta \ast \rho_3\ast \rho' \ast \rho_0$ 
for some strings $\rho_2$ 
and $\rho_3$ which are $e
$-failures in $P_\sigma(\sigma')$, and some 
finite strings $\rho$ and $\rho'$.
Note that $\Phi^{\rho_0} \supseteq \alpha \ast \sigma'$ so if 
 $\Phi^{\rho_0}$ is comparable with $\Phi^{\rho_1}$, then this implies that  
 $\rho_1\in P_\sigma(\sigma')$ 
and  $\rho_1 =\rho_2$. In this case,   
$\Psi^{\rho_1} \supseteq \beta \ast \rho_1$. 
Now because $\rho_1$ is 
incomparable with $\rho_3$ we have that  
$\Psi^{\rho_0}$ is incomparable with 
$\Psi^{\rho_1}$.

The lemma follows from an induction on the 
number of times the marker is made 
inactive because 
(\ref{reset}) does not hold. The strings 
$\rho_0$ and $\rho_1$ must extend 
different elements in some least 
$P_\sigma$, at which point the above argument holds.
\end{proof}

Given the above lemma we can  define a 
Turing functional $\Gamma$ such that 
if $\Phi^Y$ and $\Psi^Y$ 
are total, then 
$\Gamma(\Phi^Y \oplus \Psi^Y)=Y$ as follows. If for any string 
$\rho$, the main construction 
enumerates a $\Phi$-axiom $\la \rho, \alpha\ra$ and a $\Psi$-axiom 
$\la\rho, \beta\ra$  then enumerate a $\Gamma$-axiom 
$\la \alpha \oplus \beta, \rho \ra$.
\end{proof}

\section{The meet and complementation properties}
We say that a degree $\boldsymbol{c}$ satisfies the meet property if, for all  
$\boldsymbol{b}<\boldsymbol{c}$ 
there exists a non-zero 
$\boldsymbol{a}\leq \boldsymbol{c}$ with 
$\boldsymbol{b}\wedge \boldsymbol{a}=\boldsymbol{0}$. 
We say that  a degree $\boldsymbol{c}$ satisfies the 
complementation property if, for all  
non-zero $\boldsymbol{b}<\boldsymbol{c}$ there exists a non-zero 
$\boldsymbol{a}< \boldsymbol{c}$ with 
$\boldsymbol{b}\wedge \boldsymbol{a}=\boldsymbol{0}$ and 
$\boldsymbol{b}\vee \boldsymbol{a}=\boldsymbol{c}$. 

In \cite{GMS04} it was shown that all generalized 
high degrees have the complementation property. 
It remains open, however, as to whether this result is sharp. 
In particular we do not know if all GH$_2$ degrees satisfy the complementation 
property.
It is also unknown if all GH$_2$ degrees satisfy the meet property.  
In fact, we do not even know if all non-GL$_2$ degrees satisfy the 
complementation property. 

\subsection{Category}
Kumabe \cite{apal/Kumabe93} showed that every 2-generic satisfies the 
complementation property, and so 
also satisfies the meet property. The remaining questions are as to the extent to 
which this result is sharp: 

\begin{question} Do all 1-generics satisfy the complementation property? 
\end{question}

\noindent Again, the case for the meet property is also unknown: 

\begin{question} \label{q6} Do all 1-generics satisfy the meet property? 
\end{question}

\noindent We would expect a negative answer to Question \ref{q6}.

\subsection{Measure} For the case of measure, nothing is known.

\begin{question} \label{qmm} What is the measure of the degrees which satisfy 
the complementation 
property? 
How about the meet property? 
\end{question}

\noindent We would expect the answer to 
both parts of Question \ref{qmm} to be 0.

\section{The typical lower cone}

We close by considering some questions which concern 
what happens to the theory of the lower cone in the limit. For any degree 
$\boldsymbol{a}$ let $\boldsymbol{D}[\leq \boldsymbol{a}]$ denote the set of 
degrees below $\boldsymbol{a}$ with the inherited ordering relation, 
and let $\textup{Th}[\leq \boldsymbol{a}]$ be the (first order) theory of this 
structure.  
If $\phi$ is any sentence in the first order language of 
partial orders, then the set of all $A$ such that, for 
$\boldsymbol{a}=deg(A)$, $\phi\in \textup{Th}[\leq \boldsymbol{a}]$, 
is arithmetical and is therefore either meager or comeager and 
either of measure 0 or measure 1. Thus there exist $C_{\phi}$ and 
$D_{\phi}$ such that either all $C_{\phi}$-generic sets $A$ have 
$\phi\in \textup{Th}[\leq \boldsymbol{a}]$ or else all $C_{\phi}$-generic sets 
$A$ have the negation of $\phi$  in $\textup{Th}[\leq \boldsymbol{a}]$, 
and either all $D_{\phi}$-random sets $A$ have 
$\phi\in\textup{Th}[\leq \boldsymbol{a}]$ or else all 
$D_{\phi}$-random sets $A$ have the negation of $\phi$  in 
$\textup{Th}[\leq \boldsymbol{a}]$. Taking $C$ Turing above all 
$C_{\phi}$ and $D$ Turing above all $D_{\phi}$,  
we conclude that for all sufficiently generic degrees $\boldsymbol{a}$ 
and $\boldsymbol{b}$, 
$\textup{Th}[\leq \boldsymbol{a}]=\textup{Th}[\leq \boldsymbol{b}]$, 
and  for all sufficiently random degrees $\boldsymbol{a}$ and 
$\boldsymbol{b}$, 
$\textup{Th}[\leq \boldsymbol{a}]=\textup{Th}[\leq \boldsymbol{b}]$. 
Let us call these theories 
$\textup{Th}[\leq \mbox{\textup{Gen}}]$ and  
$\textup{Th}[\leq \mbox{\textup{Ran}}]$ 
respectively. We discussed earlier, that all sufficiently random degrees 
have a strong minimal cover, while all sufficiently generic degrees satisfy 
the cupping property. These are not properties which pertain to the lower cone, 
however, so the following question remains open: 

\begin{question}
Is  $\textup{Th}[\leq \mbox{\textup{Gen}}]=\textup{Th}[\leq \mbox{\textup{Ran}}]$?
\end{question} 

\noindent One would presumably expect this question to receive a negative 
answer. 

While it is clear that arithmetical randomness and genericity suffices, 
one might also ask for proof that this is the exact level required:

\begin{question} What is the exact level of 
randomness/genericity required in order to ensure that  
$\textup{Th}[\leq \boldsymbol{a}]=\textup{Th}[\leq \mbox{\textup{Ran}}]$ or  
$\textup{Th}[\leq \boldsymbol{a}]=\textup{Th}[\leq \mbox{\textup{Gen}}]$?
\end{question} 

Finally, we give some remarks on the complexity of
$\textup{Th}[\leq \boldsymbol{a}]$ for a sufficiently generic or random 
$\boldsymbol{a}$.
Greenberg and Montalb{\'a}n \cite{GreenbergM03} showed that
if the 1-generic degrees are downward dense
in an ideal $\mathcal{J}$ (that is, every nonzero $\boldsymbol{a}\in\mathcal{J}$ 
bounds
a 1-generic) then the first order true arithmetic is many-one reducible to 
the theory of $(\mathcal{J},\leq)$. Theorem \ref{th:2ranb1gen}
says that the 1-generic degrees are downward dense in the degrees below a 
2-random
degree. Therefore if $\boldsymbol{a}$ is 2-random then 
 $\textup{Th}[\leq \boldsymbol{a}]$ interpretes true arithmetic.
 The case for 2-generics is also true and was explicitly stated in 
 \cite{GreenbergM03}.
 

\begin{thebibliography}{DNWY06}

\bibitem[BDN11]{BDNGP}
George Barmpalias, Rod Downey, and Keng~Meng Ng.
\newblock Jump inversions inside effectively closed sets and applications to
  randomness.
\newblock {\em J. Symb. Log.}, 76(2):491--518, 2011.

\bibitem[BL11]{BL2011}
George Barmpalias and Andrew E.~M. Lewis.
\newblock Measure and cupping in the {T}uring degrees.
\newblock {\em Proc. Amer. Math. Soc.}, 2011.
\newblock (in press).

\bibitem[CD90]{ChoDo90}
C.T. Chong and R.G. Downey.
\newblock On degrees bounding minimal degrees.
\newblock {\em Ann. Pure Appl. Logic}, 48:215--225, 1990.

\bibitem[Cen99]{MR1720779}
D.~Cenzer.
\newblock {$\Pi\sp 0\sb 1$} classes in computability theory.
\newblock In {\em Handbook of computability theory}, volume 140 of {\em Stud.
  Logic Found. Math.}, pages 37--85. North-Holland, Amsterdam, 1999.

\bibitem[Coo04]{Cooper:04}
S.~B. Cooper.
\newblock {\em Computability theory}.
\newblock Chapman Hall, New York, 2004.

\bibitem[DGLM11]{DGLM11}
Rodney~G. Downey, N.~Greenberg, Andrew E.~M. Lewis, and Antonio Montalb\'{a}n.
\newblock Extensions of embeddings below computably enumerable degrees.
\newblock {\em Trans. Amer. Math. Soc.}, 2011.
\newblock {I}n press.

\bibitem[DH10]{rodenisbook}
R.G. Downey and D.~Hirshfeldt.
\newblock {\em Algorithmic Randomness and Complexity}.
\newblock Springer, 2010.

\bibitem[DJS96]{DJS96}
R.G. Downey, C.~Jockusch, and M.~Stob.
\newblock Array non-recursive degrees and genericity.
\newblock In {\em Computability, {E}numerability, {U}nsolvability}, volume 224
  of {\em London {M}athematical {S}ociety Lecture Note Series}, pages 93--104.
  Cambridge University Press, 1996.

\bibitem[dLMSS55]{deleeuw1955}
K.~de~Leeuw, E.~F. Moore, C.~E. Shannon, and N.~Shapiro.
\newblock {Computability by probabilistic machines}.
\newblock In C.~E. Shannon and J.~McCarthy, editors, {\em Automata Studies},
  pages 183--212. Princeton University Press, Princeton, NJ, 1955.

\bibitem[DNWY06]{DNWY}
R.G. Downey, A.~Nies, R.~Weber, and L.~Yu.
\newblock Lowness and {$\Pi^0_2$} null sets.
\newblock {\em J. Symbolic Logic}, 71:1044--1052, 2006.

\bibitem[DY06]{Downey_arithmeticalsacks}
R.G. Downey and L.~Yu.
\newblock Arithmetical {S}acks forcing.
\newblock {\em Archive for Mathematical Logic}, 45, 2006.

\bibitem[GM03]{GreenbergM03}
Noam Greenberg and Antonio Montalb{\'a}n.
\newblock Embedding and coding below a 1-generic degree.
\newblock {\em Notre Dame Journal of Formal Logic}, 44(4):200--216, 2003.

\bibitem[GMS04]{GMS04}
N.~Greenberg, A.~Montalb\'{a}n, and R.~Shore.
\newblock Generalized high degrees have the complementation property.
\newblock {\em J. Symbolic Logic}, 69:1200--1220, 2004.

\bibitem[HNS07]{MR2352724}
D.~Hirschfeldt, A.~Nies, and F.~Stephan.
\newblock Using random sets as oracles.
\newblock {\em J. Lond. Math. Soc. (2)}, 75(3):610--622, 2007.

\bibitem[Joc73]{CJ73}
C.~Jockusch.
\newblock An application of {$\Sigma^0_4$} determinacy to the degrees of
  unsolvability.
\newblock {\em J. Symbolic Logic}, 38:293--294, 1973.

\bibitem[Joc77]{Jockusch:77}
C.~Jockusch.
\newblock Simple proofs of some theorems on high degrees of unsolvability.
\newblock {\em Canad. J. Math.}, 29(5):1072--1080, 1977.

\bibitem[Joc80]{Jockusch:80}
C.~Jockusch.
\newblock Degrees of generic sets.
\newblock In F.~R. Drake and S.~S. Wainer, editors, {\em Recursion Theory: Its
  Generalizations and Applications, Proceedings of Logic Colloquium '79, Leeds,
  August 1979}, pages 110--139, Cambridge, U. K., 1980. Cambridge University
  Press.

\bibitem[Kau91]{Kautz:91}
S.~Kautz.
\newblock {\em Degrees of random sets}.
\newblock Ph.{D.} {D}issertation, Cornell University, 1991.

\bibitem[Kec95]{AK95}
A.S. Kechris.
\newblock {\em Classical descriptive set theory}.
\newblock Springer-Verlag, New York, 1995.

\bibitem[KP54]{Kleene.Post:54}
S.C. Kleene and E.~Post.
\newblock The upper semi-lattice of degrees of recursive unsolvability.
\newblock {\em Ann. of Math. (2)}, 59:379--407, 1954.

\bibitem[Ku{\v{c}}85]{MR820784}
A.~Ku{\v{c}}era.
\newblock Measure, {$\Pi\sp 0\sb 1$}-classes and complete extensions of {${\rm
  PA}$}.
\newblock In {\em Recursion theory week (Oberwolfach, 1984)}, volume 1141 of
  {\em Lecture Notes in Math.}, pages 245--259. Springer, Berlin, 1985.

\bibitem[Ku{\v{c}}93]{MR1238109}
A.~Ku{\v{c}}era.
\newblock On relative randomness.
\newblock {\em Ann. Pure Appl. Logic}, 63(1):61--67, 1993.
\newblock 9th International Congress of Logic, Methodology and Philosophy of
  Science (Uppsala, 1991).

\bibitem[Kum90]{Kuma:90}
M.~Kumabe.
\newblock A 1-generic degree which bounds a minimal degree.
\newblock {\em J. Symbolic Logic}, 55:733--743, 1990.

\bibitem[Kum91]{Kuma:91}
M.~Kumabe.
\newblock Relative recursive enumerability of generic degrees.
\newblock {\em J. Symbolic Logic}, 56(3):1075--1084, 1991.

\bibitem[Kum93a]{Kuma:93}
M.~Kumabe.
\newblock Every $n$-generic degree is a minimal cover of an $n$-generic degree.
\newblock {\em J. Symbolic Logic}, 58(1):219--231, 1993.

\bibitem[Kum93b]{apal/Kumabe93}
M.~Kumabe.
\newblock Generic degrees are complemented.
\newblock {\em Ann. Pure Appl. Logic}, 59(3):257--272, 1993.

\bibitem[Kum00]{Kuma:00}
M.~Kumabe.
\newblock A 1-generic degree with a strong minimal cover.
\newblock {\em J. Symbolic Logic}, 65(3):1395--1442, 2000.

\bibitem[Kur81]{Kurtz:81}
S.~Kurtz.
\newblock {\em Randomness and genericity in the degrees of unsolvability}.
\newblock Ph.{D.} {D}issertation, University of Illinois, Urbana, 1981.

\bibitem[Ler86]{ML86}
M.~Lerman.
\newblock Degrees which do not bound minimal degrees.
\newblock {\em Ann. Pure Appl. Logic}, 30:249--276, 1986.

\bibitem[Lew07]{AL07}
Andrew E.~M. Lewis.
\newblock {$\Pi^0_1$} classes, strong minimal covers and hyperimmune-free
  degrees.
\newblock {\em Bulletin of the London Mathematical Society}, 39(6):892--910,
  2007.

\bibitem[Lew11]{AL11}
Andrew E.~M. Lewis.
\newblock A note on the join property.
\newblock {\em Proceedings of the American Mathematical Society}, 2011.
\newblock (in press).

\bibitem[LMN07]{LMN07}
Andrew E.~M. Lewis, Antonio Montalb{\'a}n, and Andr{\'e} Nies.
\newblock A weakly 2-random set that is not generalized low.
\newblock In S.~Barry Cooper, Benedikt L{\"o}we, and Andrea Sorbi, editors,
  {\em CiE}, volume 4497 of {\em Lecture Notes in Computer Science}, pages
  474--477. Springer, 2007.

\bibitem[Mar67]{Martin:60}
D.~Martin.
\newblock Measure, category, and degrees of unsolvability.
\newblock Unpublished manuscript, 1967.

\bibitem[MS64]{MS64}
J.~Mycielski and S.~\'{S}wierczkowski.
\newblock On the {L}ebesgue measurability and the axiom of determinateness.
\newblock {\em Fund. Math.}, 54:67--71, 1964.

\bibitem[Nie09]{Ottobook}
A.~Nies.
\newblock {\em Computability and Randomness}.
\newblock Oxford University Press, 2009.

\bibitem[Odi89]{Odifreddi:89}
P.~G. Odifreddi.
\newblock {\em Classical recursion theory. {V}ol. {I}}.
\newblock North-Holland Publishing Co., Amsterdam, 1989.

\bibitem[Par77]{paris77}
J.~Paris.
\newblock Measure and minimal degrees.
\newblock {\em Ann. Math. Logic}, 11:203--216, 1977.

\bibitem[PR81]{DBLP:journals/jsyml/PosnerR81}
David~B. Posner and Robert~W. Robinson.
\newblock Degrees joining to $\mathbf{0}'$.
\newblock {\em J. Symb. Log.}, 46(4):714--722, 1981.

\bibitem[Ros75]{JR75}
J.~Rosenthal.
\newblock Nonmeasurable invariant sets.
\newblock {\em American Mathematical Monthly}, 82:488--491, 1975.

\bibitem[SW86]{SW86}
T.~Slaman and H.~Woodin.
\newblock Definability in the {T}uring degrees.
\newblock {\em Illinois J. Math.}, 30:320--334, 1986.

\bibitem[Yat70]{Yates:70*1}
C.E.M. Yates.
\newblock Initial segments of the degrees of unsolvability, part {II}: Minimal
  degrees.
\newblock {\em J. Symbolic Logic}, 35:243--266, 1970.

\bibitem[Yat76]{Yates:76}
C.E.M. Yates.
\newblock {B}anach-{M}azur games, comeager sets and degrees of unsolvability.
\newblock {\em Math. Proc. Cambridge Philos. Soc.}, 79:195--220, 1976.

\end{thebibliography}

\end{document}